\documentclass[oneside,notitlepage,11pt]{amsart}

\usepackage[T1]{fontenc}

\usepackage{amssymb}
\usepackage{enumerate}
\usepackage[leqno]{amsmath}
\usepackage{amsfonts}
\usepackage{amsopn}
\usepackage{amstext}
\usepackage{amsthm}
\usepackage{easyReview}

\usepackage[all]{xy}
\newdir{ >}{{}*!/-9pt/@{>}}

\usepackage[colorlinks]{hyperref}

\newenvironment{pf}{\begin{proof}}{\end{proof}}

\newcommand{\eps}{\varepsilon}
\renewcommand{\phi}{\varphi}
\renewcommand{\rho}{\varrho}

\newcommand{\cl}{\operatorname{cl}}

\newcommand{\dist}{\operatorname{dist}}

\newtheorem{tw}{Theorem}[section]

\newtheorem{lm}[tw]{Lemma}
\newtheorem{prop}[tw]{Proposition}
\newtheorem{cor}[tw]{Corollary}
\newtheorem{claim}[tw]{Claim}
\theoremstyle{definition}
\newtheorem{df}[tw]{Definition}
\newtheorem{rem}[tw]{Remark}

\newtheorem{pyt}[tw]{Question}

\theoremstyle{remark}

\newcommand{\U}{\mathbb U}

\newcommand{\R}{\mathbb R}

\newcommand{\N}{\mathbb N}

\newcommand{\Q}{\mathbb Q}
\newcommand{\G}{\mathbb U}

\title{Topological size of the set of universal and ultrahomogeneous retractions on the Urysohn space}
\author{
{\sc J. B\k{a}k, J. Garbuli\'nska-W\k{e}grzyn, M. Pop\l{}awski}\\ \\
}
\address{J. B\k{a}k: 
		(ORCID 0000-0001-8027-7226): 
		Mathematics Department, 
		Jan Kochanowski University in Kielce,
		Uniwersytecka 7, 25-406 Kielce, Poland}
	\email{judyta.bak@ujk.edu.pl}
	
	\address{J. Garbuli\'nska-W\k{e}grzyn
		(ORCID 0000-0001-7217-2002): 
		Mathematics Department, 
		Jan Kochanowski University in Kielce,
		Uniwersytecka 7, 25-406 Kielce, Poland}
	\email{jgarbulinska@ujk.edu.pl}
	
	\address{M. Pop\l{}awski
		(ORCID 0000-0002-2725-9675): Institute of Mathematics,
         Lodz University of Technology, al. Politechniki 8,
         93-590 \L\'od\'z,
          Poland}
         \email{michal.poplawski.m@gmail.com}

\begin{document}

\maketitle
\vspace{-2cm}
\begin{abstract}
In this paper, we investigate the set $\mathcal{U}(\U)$ of universal and ultrahomogeneous $1$-Lipschitz retractions acting on the Urysohn space as the subspace of the space $\mathcal{R}(\U)$ of all $1-$Lipschitz retractions defined on the Urysohn space. Especially, we study Borel complexity and density $\mathcal{U}(\U)$ in $\mathcal{R}(\U).$ In order to do that, we introduce a new extension property $(UR^*)$ that is equivalent to the universality and ultrahomogeneity of a retraction, and a new pointwise retract topology.
\end{abstract}



\section{Introduction}
The space $\mathbb{X}$ is \emph{universal}, if it contains an isometric copy of every complete separable metric space.
Urysohn in \cite{ur} constructed a complete separable metric space $(\mathbb{U},d_{\U})$ that is universal. A more famous example of a universal space is $C([0,1])$.

A map $f : X\to Y$ is a $1-$Lipschitz if its Lipschitz constant $$\text{Lip}(f) = \sup_{x\neq y}\frac{d(f(x), f(y))}{d(x, y)} \leq 1.$$

Urysohn’s universal metric space $\U$ was characterized in \cite{ur} as being, up to isometry, the
unique Polish metric space with the following \emph{extension property}:
\begin{itemize}
	\item[(U)] for any finite metric space $X$ and any isometry $f:X_0 \to \U$, where $X_0\subseteq X$,  there exists an isometry
	$\bar{f}:X\to \U$ extending $f$ ($\bar{f}\restriction _{X_0}=f$).

\end{itemize}
An analogous extension property holds for the rational Urysohn space $(\U_0, d_{\U_0})$, equipped with the metric $d_{\U_0}$ inherited from $\mathbb{U}$, as well as for any rational metric space $X$.




Michal Doucha investigated in \cite{d} the concepts of universality and ultrahomogeneity in a context of retractions defined on the Urysohn space.
\begin{df}
 We say that a $1$-Lipschitz retraction $U\colon \U \to U[\U]$ is \emph{universal}, if for every separable metric space $X$ and for every $1$-Lipschitz retraction $R \colon X \to R[X]$ there exists an isometry $i\colon X \to \U$ such that for every $x\in X$ we have
	$$i\circ R(x)=U\circ i(x) \quad \mbox{ and }\quad \dist(x,R[X])=\dist (i(x), U[\U]).$$
\end{df}

\begin{df}{[Remark 2.9, \cite{d}]}\label{auto}
$(\U, U, U[\U])$ is \emph{ultrahomogenous} if satisfy the following property:
let $A_1,A_2\subseteq \U$ be two finite subsets that are isomorphic, i.e. there
exists an isometry $i: A_1\to A_2$ such that
\begin{itemize}
	\item for every $x \in A_1$ we have $i \circ U(x) =U \circ i(x)$,
	\item for every $x \in A_1$ we have $\mbox{dist}(x, U[\U]) = \mbox{dist}(i(x), U[\U])$.
\end{itemize}
Then $i$ extends to an autoisometry $I \colon \U \to \U$ with
\begin{itemize}
	\item for every $x \in \U$ we have $I \circ U(x) =U \circ I(x)$,
	\item for every $x \in \U$ we have $\mbox{dist}(x, U[\U]) = \mbox{dist}(I(x), U[\U])$.
\end{itemize}

\end{df}

In the further part of the article, the triple $((A, d), r, p)$ denotes a metric space $(A,d)$, $1-$Lipschitz retraction $r\colon A\to A$ and $1-$Lipschitz map $p\colon A\to [0,\infty)$ with \linebreak $r[A]=\{x \in A \colon p(x)=0\}$. If the metric on 
$A$ is clear from the context, we denote it by $(A,r,p)$. If the metric on $A$ and the map $p$ attain rational values, we call this triple rational, if the set $A$ is finite, we will say it is a finite triple. By a triple $(\U,U,D_U)$ we understand Urysohn space $\U$, $1-$Lipschitz retraction $U \colon \U \to U[\U]$ and $1-$Lipschitz function $D_U \colon \U \to [0,\infty)$ defined by $D_U(x)=\dist(x,U[\U])$ for any $x \in \U.$ If the function 
$D_U$ is clear from the context, we will denote it by $D$.

M. Doucha proved in \cite{d} that there exists an universal, ultrahomogeneous retraction $U$ defined on the Urysohn space on the subspace $U[\U] \subsetneq \U$ isometric to $\U$, satisfying the following condition [\cite{d}, Lemma 2.10, One-point extension property]
\begin{enumerate}
\item[($UR$)] 
$(\U, U, D)$, where $D(x)=\mbox{dist}(x, U[\U])$, has one-point extension, i.e for every triple $(A, r,p)$ and every embedding $i : A\to (\U, U, D )$ such that $U \circ i(x) = i \circ r(x)$ and $D \circ i(x) = p(x)$ for all $x\in A$, for every $B=A\cup \{b\}$, $(B,r',p')$, $r'\restriction A=r$, $p'\restriction A=p$, there exists an extension $i':B\to (\U, U, D)$ of $i$ such that for every $x\in B$  we have $$U \circ i'(x) = i' \circ r'(x)\quad \mbox{ and } \quad D \circ i'(x) = p'(x).$$
\end{enumerate}
Let us formulate a similar condition
\begin{itemize}
\item[(UR$^{\star}$)] \label{UR2}

for every $\eps >0$, any finite triple $(A, r,p)$ every  isometric embedding $i : A\to (\U_0,U, D )$ such that $$d_{\U_0}(U \circ i(x), i \circ r(x))<\eps \quad \mbox{ and } \quad |D \circ i(x) - p(x)|<\eps,$$ and every rational triple $(B,r',p')$, where $B=A\cup \{b\}, b\notin A$, $r'\restriction A=r$, $p'\restriction A=p$, there exists an isometric embedding {$i':B\to (\U_0, U, D)$} such that for every $x \in A$ we have $d_{\U}(i(x),i'(x))< 2\eps$  and for every $x \in B$ we have
$$d_{\U}(U \circ i'(x), i' \circ r'(x))<5\eps \quad \text{ and }  \quad |D \circ i'(x) - p'(x)|<3\eps.$$

\end{itemize}
In our paper we get the following result
\begin{tw}
	Let  $U\colon \U \to U[\U]$ be $1-$Lipschitz retraction. TFAE:
	\begin{enumerate}[(i)]
	\item $U$ satisfies $(UR)$;
	\item $U$ is universal and ultrahomogeneous;
    \item $U$ satisfies $(UR^{*})$.
	\end{enumerate}
\end{tw}

Emphasize that $(UR)\Leftrightarrow(UR^{*})$ is inspired by analogous results on different extension properties for the Urysohn space and Gurari\u\i\ space. 

Namely, there is a comment in [\cite{WK}, p.58] that a condition
\begin{itemize}
	\item[(U)] for any finite metric space $X$ and any isometry $f:X_0 \to \U$, where $X_0\subseteq X$,  there exists an isometry
	$\bar{f}:X\to \U$ extending $f$ ($\bar{f}\restriction _{X_0}=f$).

\end{itemize} 
is equivalent to
\begin{itemize}
	\item[(U')] for any finite metric space $X$ and any isometry $f:X_0 \to \U$, where $X_0\subseteq X$,  then for any $\eps>0$ exists an isometry
	$\bar{f}:X\to \U$ with $d(f(x),\bar{f}(x))<\eps$ for any $x \in X_0.$

\end{itemize}

The second example concerns Gurari\u\i\ space i.e. the unique, up to a linear isometry, separable Banach space $\mathbb{G}$ satisfying the extension property: 
\begin{itemize}
	\item[(G)]
for every $\eps>0$, for every finite-dimensional spaces $X_0 \subseteq X$, for every linear isometric embedding $f_0:X_0\to \mathbb{G}$ there exists a linear $\eps$-isometric embedding $f:X \to \G$ such that $f\restriction  { X_0} = f_0$, where
\emph{$\eps$-isometric embedding} is a linear operator $f$ satisfying
$$(1-\eps)\|x\| \leq \|{f(x)}\| \leq (1+\eps)\| x\|$$
for every $x\in X$.
\end{itemize}
It can be shown (see \cite{WK}) that $(G)$ is equivalent to
\begin{itemize}
	\item[(G')]
for every $\eps>0$, for every finite-dimensional spaces $X_0 \subseteq X$, for every linear isometric embedding $f_0:X_0\to \mathbb{G}$ there exists a linear isometric embedding $f:X \to \G$ such that $\|f(x)-f_0(x)\|<\eps$ for any $x \in X_0.$
\end{itemize}
Let us introduce an example that motivates our research.  W. Kubi\'s in \cite{RW}  showed that $\mathbb{G}$ is the universal ultrahomogeneous object obtained as a Fra\"{\i}ss\'{e} limit. In \cite{BG} is shown that the family $\{I \circ \Omega \circ J^{-1} \colon I,J \colon \mathbb{G} \to  \mathbb{G} \mbox{ are autoisometries}\}$ forms a dense $G_{\delta}$ set in the space of all $1-$Lipschitz operators $\mathbb{G} \to \mathbb{G}$ endowed with the strong operator topology, when $\Omega \colon \mathbb{G} \to \mathbb{G}$ is $1-$Lipchitz universal operator constructed in \cite{GK}, obtained also as a Fra\"{\i}ss\'{e} limit.

Inspired by the above example we will define respective \emph{pointwise retract convergence} topology $\tau_{pr}$ on the set 
$$
\mathcal{R}(\U)=\{ R:\U\to \U \colon  R \text{ is $1-$Lipschitz retraction}\}
$$ 
by describing a basic neighbourhood of $U \in \mathcal{R}(\U)$ as
$$W(U)_{X,\eps}=\{R \in \mathcal{R}(\U) \colon d_{\U}(R(x),U(x))<\eps \mbox{ and } |D_U(x)-D_R(x)|<\eps \mbox{ for any } x \in X\},$$ where $X=\{x_1,x_2,\ldots,x_n\} \subset \U, \ \eps>0$ are fixed. 
In this topology 
\begin{tw}
The set \begin{align*}
\mathcal{U}(\U)=\{ U \in \mathcal{R}(\U) \colon  U \text{ is universal, ultrahomogeneous}\}
\end{align*} 
is a dense $G_{\delta}$ subset of $\mathcal{R}(\U).$
\end{tw}
We investigate the density and Borel complexity of the set $\mathcal{U}(\U)$ in $\mathcal{R}(\U)$ equipped both with the pointwise convergence topology and with the uniform convergence topology.

\section{Equivalent conditions of being universal and ultrahomogenous retraction}

Inspired by the result of~\cite{WK}, let us consider the following infinite game for two players, \emph{Eve} and \emph{Adam}.
Namely, Eve plays with finite triples $(X_{2k},r_{2k},p_{2k}), \ k \geq 0$ and
Adam plays with finite triples $(X_{2k+1},r_{2k+1},p_{2k+1}), \ k \geq 0$ such that $X_{k+1} \supseteq X_k$ and $r_{k+1}, p_{k+1}$ extends $r_k, p_k$ for any $k \geq 0$.
After infinitely many steps, we obtain a chain of finite triples $(X_k,r_k,p_k)_{k \in \omega}$. Let $r_\infty:X_\infty\to r_\infty[X_\infty]$ denote the unique $1-$Lipschitz retraction defined on the completion $X_\infty$ of $\bigcup_{k\in \omega} X_k$ such that $r_\infty \restriction X_k = r_k$ for every $k\in \omega.$ Similarly, let us define the unique $1-$Lipschitz function $p_\infty:X_\infty\to \R_0^+$ with  $p_\infty \restriction X_k =p_k$ for every $k\in \omega.$ Then $r_\infty[X_\infty]$ is the completion of $\bigcup_{k\in \omega} r_k[X_k]$  and $p_\infty(x)= \mbox{dist}(x,r_\infty[X_\infty])$ for any $x\in X_{\infty}$. 



\begin{tw}\label{URuu}

	Let  $U\colon \U \to U[\U]$ be $1-$Lipschitz retraction. TFAE:
	\begin{enumerate}[(i)]
	\item $U$ satisfies $(UR)$;
	\item $U$ is universal and ultrahomogeneous.
	\end{enumerate}
\end{tw}

	

\begin{proof}
"$(i) \Rightarrow (ii) $"
Let us fix a triple $(\U, U, D)$ satisfying condition $(UR)$. \textbf{We will show that $U$ is universal.} Adam strategy can be described as follows.

Fix a countable set $\{x_i\}_{i\in \omega}$ which is dense in the space $\U$, especially countable set $\{U(x_i)\}_{i\in \omega}$ is dense in $U[\U]$. Let $(X_0, r_0,p_0)$ be the first move of Eve. For any point $u\in r_0[X_0]$ we can find point $v\in U[\U]$ and isometry $j:\{u\}\to \{v\}$, $j(u)=v$, such that $U\circ j(u)=U(v)=v=j(u)=j(r_0(u))$ and $p(u)=0=\mbox{dist}(v,U[\U])$. Now using $|X_0|-1$-times condition $(UR)$ we get isometry $i_0:X_0\to \U$ such that $U\circ i_0(x)=i_0\circ r_0(x)$ and $p_0=\mbox{dist}(x,U[\U])$ Adam defines a triple $(X_1,r_1,p_1)$ in the following way:

\begin{itemize}
	\item if $x_0\in f_0[X_0]$ then $(X_1,r_1,p_1)=(X_0,r_0,p_0)$.
	\item otherwise let $y_0$ be such a point that  $d_{X_1}(x,y_0)=d_\U(f_0(x),x_0)$  for $x\in X_0$. Let us take a point $b_0=U(x_0)$($b_0=x_0$ or not) and let $a_0$ be such that 
	$$d_{X_1}(x,a_0)=d_\U(f_0(x),b_0) \ \text{and} \ d_{X_1}(y_0,a_0)=d_\U(x_0,b_0)$$
	for $x\in X_0$. Let us consider a metric space $X_1=X_0\cup\{y_0,a_0\}$. Put
	$$r_1\restriction X_0=r_0, r_1(y_0)=a_0, r_1(a_0)=a_0,$$
	$$p_1\restriction X_0=p_0, p_1(y_0)=\mbox{dist}(x_0,U[\U], p_1(a_0)=0.$$ 
\end{itemize}
We define an isometry 
$$f_1\restriction X_0=f_0, \ f_1(y_0)=x_0, \ f_1(a_0)=b_0.$$ 
This ensures that $f_1\circ r_1=U\circ f_1$ and $p_1(x)=\mbox{dist}(f_1(x),U[\U])$.

Suppose $n=2k>0$ and $(X_n, r_n,p_n)$ was the last move of Eve, $n\in \N$. We assume that triple $(X_{n-1}, r_{n-1},p_{n-1})$ and isometry $f_{n-1}:X_{n-1}\to \U$ has already been fixed, $\{x_0, \dots, x_{k-1}\}\subseteq f_{n-1}[X_{n-1}]\subseteq \U$. We know that $X_{n-1}\subseteq X_n$ and $r_n$ extends $r_{n-1}$, and $X_n\setminus X_{n-1}=\{a_1,\dots a_l\}$, then $l$-times using condition (UR) we get isometry $f_n:X_n\to \U$ such that $f_n\restriction X_{n-1}=f_{n-1}$, $f_n[X_n]\subseteq \U$ and $U\circ f_{n}(x)=f_{n}\circ r_{n}(x)$ and $p_n(x)=\text{dist}(f_{n}(x),U[\U])$ for every $x\in X_n$.  

Adam defines a triple $(X_{n+1},r_{n+1},p_{n+1})$ in the following way:
If $x_n\in f_n[X_n]$ then $(X_{n+1},r_{n+1},p_{n+1})=(X_n,r_n,p_n)$. Otherwise

Let $y_k$ be such a point that  $d_{X_{n+1}}(x,y_k)=d_\U(f_n(x),x_k)$  for $x\in X_n$. Let us take a point $b_k=U(x_k)$($b_k=x_k$ or not) and let $a_k$ be such that 
$$d_{X_{n+1}}(x,a_k)=d_\U(f_n(x),b_k) \ \text{and} \ d_{X_{n+1}}(y_k,a_k)=d_\U(x_k,b_k)$$
for $x\in X_n$. Let us consider a metric space $X_{n+1}=X_n\cup\{y_k,a_k\}$. Put
$$r_{n+1}\restriction X_n=r_n, r_{n+1}(y_k)=a_k, r_{n+1}(a_k)=a_k,$$
$$p_{n+1}\restriction X_n=p_n, p_{n+1}(y_k)=\mbox{dist}(x_k,U[\U]), p_{n+1}(a_k)=0.$$ 
We define an isometry 
$$f_{n+1}\restriction X_n=f_n, \ f_{n+1}(y_k)=x_k, \ f_{n+1}(a_k)=b_k.$$ 
This ensures that $f_{n+1}\circ r_{n+1}=U\circ f_{n+1}$ and $p_{n+1}(x)=\mbox{dist}(f_{n+1}(x),U[\U])$.

We will check that $p_{n+1}$ defined as above is $1-$Lipschitz.  	Let us remains that $\mbox{dist}(x,U[\U])$ is a $1-$Lipschitz function, for $x\in \U$.

 Take $x,y\in  X_{n+1}.$ Then we have  $$|p_{n+1}(x)-p_{n+1}(y)|=|\mbox{dist}(f_{n+1}(x),U[\U])-\mbox{dist}(f_{n+1}(y),U[\U])| $$ $$ \leq d_{\U}(f_{n+1}(x),f_{n+1}(y))=d_{X_{n+1}}(x,y).$$
	The other cases are easier.

This ensures that $f_{n+1}\circ r_{n+1}=U\circ f_{n+1}$ and 
$p_{n+1}(x)=\mbox{dist}(f_{n+1}(x),U[\U])$

Let $X_{\infty}$ be the completion of $\bigcup\{X_n: {n\in \omega}\}$ and $r_{\infty}[{X_{\infty}}]$ be the completion of $\bigcup\{r_n[{X_n}]:{n\in \omega}\}$.
Let $\{r_n :X_n\to r_n[{X_n}]\}_{n\in \omega}$ be the sequence of $1-$Lipschitz retractions between finite metric spaces and $\{p_n :X_n\to \R_{0}^+\}_{n\in \omega}$ be the sequence of $1-$Lipschitz resulting from a fixed play, when Adam was using his strategy with $r_{n+1}\restriction X_n=r_n, \ p_{n+1} \restriction X_n=p_n$ for all $n \in \N$. Then for any $x \in \bigcup_{n \in \N} X_n$ the sequence $(r_n(x))_{n \in \N}$ is eventually constant and consequently convergent. Then we can define $r \colon \bigcup_{n \in \N} X_n \to r_\infty[X_\infty]$ by $r(x)=r_n(x)$ if $x \in X_n.$ By definition $r$ is $1-$Lipschitz. Since $r$ is $1-$Lipschitz, then $r$ is uniformly continuous into complete metric space $r_\infty[X_\infty].$ As a result, there is exactly one uniformly continuous extension $r_\infty:X_\infty\to r_\infty[X_\infty]$ of $r$, which is $1-$Lipschitz, too. Similarly, we construct $p_\infty:X_\infty\to \R_0^+$ as a unique uniformly continuous extension of the pointwise limit of the $1-$Lipschitz functions $p_n, \ n \in \N.$

Moreover, Adam get a sequence $\{f_n : X_n \to  \U\}_{n\in \N}$ of isometric embeddings such that $f_{n}\restriction X_{n-1}=f_{n-1}$ for each $n\in \omega$. Taking $f=\bigcup_{n\in \omega} f_n$ we get isometry $f:\bigcup_{n\in \omega} X_n\to \U$. Since $f$ is isometry, it can be extended to isometry $\bar{f}:X_\infty\to \U$. Since for every $i\in \omega$ we have that $p_i$ is $1-$Lipschitz function, $r_i$ $1-$Lipschitz retraction, then the extension satisfies $U \circ f(x) = f \circ r_\infty(x)$ and
$\mbox{dist}(x, r_\infty[X_\infty]) = \mbox{dist}(f(x), U[\U])$ for each $x \in V_\infty$.  The assumption $x_{n+1}\in f_{2n+1}[X_{2n+1}]$, guarantee that $f_{2n+1}[X_{2n+1}]$ contains points $x_0,\dots,x_{n}$, this implies that $X_\infty$ is dense in $\U$, especially $R[X]$ is dense in $U[\U]$. 

\textbf{Now, we will check that $U$ is ultrahomogeneous.}

Fix a finite triple $(A,U,D)$, and an isometric embedding $i:A\to \U$ such that $U\circ i(x)=i\circ U(x)$ and $\mbox{dist}(x,U[\U])=\mbox{dist}(i(x),U[\U])$ for any $x\in A$.

We will define sequences of finite metric spaces $\{A_n\}_{n\in \omega}$,  $\{B_n\}_{n\in \omega}$ and isometric embeddings $\{i_n:A_n\to \U\}_{n\in \omega}$, $\{j_n:B_n\to \U\}_{n\in \omega}$ such that  for any $n\in \omega$:
\begin{enumerate}[(i)]
	\item $A_n\subseteq A_{n+1}$ and $B_n \subseteq B_{n+1}$,
	\item $i_{n} \colon A_{n}  \to B_{n}, \ j_{n}=i_{n}^{-1} \colon B_{n} \to A_{n} $ are isometries with $i_n\restriction A_{n-1}=i_{n-1}$	and $j_n\restriction B_{n-1}=j_{n-1}$,	
	\item $U\circ i_n(x)=i_n\circ U(x)$ for any $x\in A_n$ and $U\circ j_n(y)=j_n\circ U(y)$ for any $y\in B_n$,
	\item $\mbox{dist}(x,U[\U])=\mbox{dist}(i_n(x),U[\U])$ for any $x\in A_n$ and $\mbox{dist}(y,U[\U])=\mbox{dist}(j_n(y),U[\U])$ for any $y\in B_n$,
	\item $\overline{\bigcup_{n\in \omega} A_n}=\U$ and $\overline{\bigcup_{n\in \omega} B_n}=\U$.
\end{enumerate} 
Let us enumerate $\U_0=\{e_n \colon n \in \N\}$ and put $E=\U_0 \cup U[\U_0].$

Taking $A_0:=A$, $B_0=i[A]$ and $i_0:=i$, $j_0:=i^{-1}$ we get first inductive step. Assume now that inductive condition holds for some $2n$, this mean we have finite triples \linebreak $(A_{2n},U \restriction A_{2n},D\restriction A_{2n})$, $(B_{2n},U \restriction B_{2n},D\restriction B_{2n})$, and isometric embeddings $i_{2n}:A_{2n}\to B_{2n}$, $j_{2n}:B_{2n}\to A_{2n}$ such that $U\circ i_{2n}(x)=i_{2n}\circ U(x)$, $U\circ j_{2n}(y)=j_{2n}\circ U(y)$ and $\mbox{dist}(x,U[\U])=\mbox{dist}(i_{2n}(x),U[\U])$, $\mbox{dist}(y,U[\U])=\mbox{dist}(j_{2n}(y),U[\U])$ for any $x\in A_{2n}$, $y\in B_{2n}.$ We consider finite triples $$(A_{2n+1},U \restriction A_{2n+1},D\restriction A_{2n+1}), \ (B_{2n+1},U \restriction B_{2n+1},D\restriction B_{2n+1}), $$ $$ (A_{2n+2},U \restriction A_{2n+2},D\restriction A_{2n+2}), \ (B_{2n+2},U \restriction B_{2n+2},D\restriction B_{2n+2}).$$

Put $A_{2n+1}=A_{2n} \cup \{e_{n},U(e_{n})\}$ and by $(UR)$ we find an isometry $i_{2n+1}:A_{n+1}\to \U$ such that $U\circ i_{2n+1}(x)=i_{2n+1}\circ U(x)$ and $\mbox{dist}(x,U[\U])=\mbox{dist}(i_{2n+1}(x),U[\U])$ for any $x\in A_{2n+1}$ and $i_{2n+1}\restriction A_{2n}=i_{2n}$. We define $B_{2n+1}=i_{2n+1}[A_{2n+1}]$ and $j_{2n+1}=i_{2n+1}^{-1} \colon B_{2n+1} \to A_{2n+1}.$ Set $B_{2n+2}=B_{2n+1} \cup \{e_{n},U(e_n)\}$ and by $(UR)$ we find an isometry $j_{2n+2}:B_{n+2}\to \U$ such that $U\circ j_{2n+2}(y)=j_{2n+2}\circ U(y)$ and $\mbox{dist}(y,U[\U])=\mbox{dist}(j_{2n+2}(y),U[\U])$ for any $y\in B_{2n+2}$ and $j_{2n+2}\restriction B_{2n+1}=j_{2n+1}$. We put $A_{2n+2}=j_{2n+2}[B_{2n+2}]$ and $i_{2n+2}=j_{2n+2}^{-1} \colon A_{2n+2} \to B_{2n+2}.$ It can be easily seen that $i_{2n+2} \restriction A_{2n+1}=i_{2n+1}$ and $j_{2n+1} \restriction A_{2n+1}=j_{2n}.$


Let $A_{\infty}=\overline{\bigcup\{A_n: {n\in \omega}\}}$ and $B_{\infty}=\overline{\bigcup\{B_n: {n\in \omega}\}}$. Since $\U_0 \subset \bigcup\{A_n: {n\in \omega}\} \cap \bigcup\{B_n: {n\in \omega}\}$ we have $A_{\infty}=B_{\infty}=\U.$
We get sequences $\{i_n : A_n \to  \U\}_{n\in \N}$, $\{j_n : B_n \to  \U\}_{n\in \N}$ of isometric embeddings such that $i_{n}\restriction A_{n-1}=i_{n-1}$,  $j_{n}\restriction B_{n-1}=j_{n-1}$ for each $n\in \omega$. Taking $i=\bigcup_{n\in \omega} i_n$, $j=\bigcup_{n\in \omega} j_n$ we get isometry $i:\bigcup_{n\in \omega} A_n\to \U$, $j:\bigcup_{n\in \omega} B_n\to \U$ Since $i, j$ are isometries, them can be extended to isometries $\bar{i}:A_\infty\to \U$, $\bar{j}:B_\infty\to \U$. Since for every $i\in \omega$ we have $1-$Lipschitz function and $1-$Lipschitz retraction, then the extension satisfies $U \circ \bar{i}(x) = \bar{i} \circ U(x)$,  $U \circ \bar{j}(x) = \bar{j} \circ U(x)$ and
$\mbox{dist}(x, U[\U]) = \mbox{dist}(\bar{i}(x), U[\U])$, $\mbox{dist}(y, U[\U]) = \mbox{dist}(\bar{j}(y), U[\U])$ for each $x \in A_\infty$, $y \in B_\infty$.  

$(ii) \Rightarrow (i)$
Fix a finite triple $(B,r,p)$, with a metric space $B=A\cup \{b\}$, where $b\notin A$. Fix an isometric embedding $i:A\to \U$ such that $U\circ i(x)=i\circ r(x)$ and $p(x)=\mbox{dist}(i(x),U[\U])$ for any $x\in A$.

Using the universality of $U$ we find an isometric embedding $j:B\to \U$ such that $U\circ j(x)=j\circ r(x)$ and $p(x)=\mbox{dist}(j(x),U[\U])$ for any $x\in B$.

Note that we get two finite triples $(i[A],U,D)$ and $(j[B],U,D)$. Now we define an isometric embedding $k:j[A]\to i[A]$ as $k(j(a))=i(a)$ for any $a\in A$. 

Observe that for any $x\in j[A], x=j(a)$ we have:
$$U\circ k(x)=U(k(j(a)))=U(i(a))=i(r(a))=k(j(r(a))=k(U(j(a)))=k\circ U(x)$$ and
$$\mbox{dist}(k(x), U[\U])=\mbox{dist}(k(j(a)), U[\U])=\mbox{dist}(i(a), U[\U])=\mbox{dist}(j(a), U[\U])=\mbox{dist}(x, U[\U]).$$
Using ultrahomogeneity, we can extend the isometric embedding $k:A\to \U$ to autoisometry $K:\U\to \U$ such that $U\circ K(x)=K\circ U(x)$ and $\mbox{dist}(x, U[\U])=\mbox{dist}(K(x), U[\U])$.

Let $i':=K \circ j$, then $i'$ is an isometric embedding such that $i'\restriction A=k\circ j=i$ and for point $b$ we have $$U\circ i'(b)=U(K(j(b)))=K(U(j(b)))=K(j(r(b)))=i'\circ r(b)$$ and $$\mbox{dist}(i'(b), U[\U])=\mbox{dist}(K(j(b)), U[\U])=\mbox{dist}(j(b), U[\U])=p(b).$$

\end{proof}
The following result is a consequence of Fra\"{\i}ss\'{e} theory. For the reader’s convenience, we include a constructive proof.

\begin{tw}\label{retauto}
	Let $U \colon \U \to \U$ be a universal and ultrahomogeneous $1-$Lipschitz retraction.  Suppose that $T \colon \U \to \U$ is $1-$Lipschitz retraction. The following conditions are equivalent:
	\begin{enumerate}[(i)]
		\item $T$ is universal and ultrahomogeneous;
		\item there is an autoisometry $I \colon \U \to \U$ such that 
		$$U \circ I=I \circ T \quad \mbox{ and }  \dist(x,T[\U])=\dist(I(x),U[\U]) \mbox{ for all } x \in \U;$$
        \item there is an autoisometry $I \colon \U \to \U$ such that 
		$$U \circ I=I \circ T \mbox{ for all } x \in \U.$$
	\end{enumerate}
\end{tw}

\begin{proof}
	$"(i) \Rightarrow (ii)"$
	
	Enumerate $\U_0=\{x_n \colon n \in \N \}$ and set $U_n=\{x_i \colon 1 \leq i \leq n\}.$ Let us denote $D_T=\dist(\cdot,T[\U]) \colon \U \to \R, \ D_U=\dist(\cdot,U[\U]) \colon \U \to \R.$ We will build a sequence $(i_n)_{n \in \N}$ of isometries and sequences $((X_n,r_n,p_n))_{n \in \N}, \ ((Y_n,r'_n,p'_n))_{n \in \N}$ of triples such that $i_n \colon X_n \to Y_n$ is surjective and $U_n \subset X_n \subset X_{n+1}, \ U_n \subset Y_n \subset Y_{n+1}$ for all $n \in \N.$  Fix $z \in T[\U], \ y \in U[\U]$ and define an isometry $i_0 \colon \{z\} \to \{y\}, \ i_0(z)=y.$ 
		Let us define $X_0=\{z\}.$ We define a triple 
	$(X_0,r_0,p_0)$ by $r_0=T \restriction X_0, \ p_0=D_T \restriction X_0.$
	
	Then we have 
	$$U \circ i_0(z)=U(y)=y=i_0(z)=i_0 \circ T(z)=i_0 \circ r_0(z)$$ and
	$$p_0(z)=\dist(z,T[\U])=0=\dist(y,U[\U])=\dist(i_0(z),U[\U])=D_U \circ i_0(z).$$

	Suppose that for some $n \in \N$ we defined $(X_n,r_n,p_n)$ and $i_n$ such that $$r_n=T \restriction X_n, \ p_n=p \restriction X_n$$ satisfying
	$$U \circ i_n(x)=i_n \circ r_n(x) \mbox{ and } p_n(x)=D_U \circ i_n(x)$$ for all $x \in X_n.$
	 We define a triple $(Z_n,U \restriction Z_n,D_U \restriction Z_n)$ for $Z_n=i_n[X_n].$ Note that $i_n^{-1} \colon Z_n \to X_n$ satisfies
	 $$i_n^{-1} \circ U(x)=T \circ i_n^{-1}(x) \mbox{ and } \dist(i_n^{-1}(x),T[\U])=\dist(x,U[\U])$$ for any $x \in Z_n.$ Let us define a triple $(Y_n,r'_n,p'_n)$ by 
	 $$Y_n=Z_n \cup \{x_n,U(x_n)\}, \ r_n'=U \restriction Z_n, \ p_n'=D_U \restriction Z_n.$$
	 We use $(UR)$ to find an isometry $j_{n+1}^{-1} \colon Y_n \to j_{n+1}^{-1}[Y_n]$ extending $i_n^{-1}$ with
	 $$j_{n+1}^{-1} \circ r_n'(x)=T \circ j_{n+1}^{-1}(x) \mbox{ and } p_n'(x)=D_T \circ j_{n+1}^{-1}(x)$$ for any $x \in Y_n.$
	 In other words, we found an isometry $j_{n+1} \colon j_{n+1}^{-1}[Y_n] \to Y_n$  ($j_{n+1}$ is an inverse of $j_{n+1}^{-1}$) extending $i_n$ with
	 $$U \circ j_{n+1}(x)=j_{n+1} \circ T(x) \mbox{ and } \dist(j_{n+1}(x),U[\U])=\dist(x,T[\U])$$ for any $x \in j_{n+1}^{-1}[Y_n].$
     Now, we define a triple $(X_{n+1},r_{n+1},p_{n+1})$ by 
     $$X_{n+1}=j_{n+1}^{-1}[Y_n] \cup \{x_{n+1}, T(x_{n+1})\}, \ r_{n+1}=T \restriction X_{n+1}, \ p_{n+1}=D_T \restriction X_{n+1}.$$
	Again, we use  $(UR)$ to find an isometry $i_{n+1} \colon X_{n+1} \to i_{n+1}[X_{n+1}] $ extending $j_{n+1}$ satisfying
	 $$U \circ i_{n+1}(x)=i_{n+1} \circ r_{n+1}(x) \mbox{ and } p_{n+1}(x)=D_U \circ i_{n+1}(x)$$ for any $x \in X_{n+1}.$ 
	 
	 Now, for any $n \in \N$ we define an autoisometry $I_n \colon \U \to \U$ extending $i_n \colon X_n \to Y_n.$ Note that $U_n \subset X_{n} \cap Y_{n+1}.$ Moreover we have $I_n \circ I_n^{-1}=I_n^{-1} \circ I_n=\textrm{Id}$ and pointwise limits $I=\lim I_n, \ I^{-1}=\lim I_n^{-1}$ are isometries. Then $I \colon \U \to \U$ is an autoisometry. Now, fix $x \in \U_0$ and find $k \in \N$ with $x \in U_k.$ Then, we have 
	 $$U \circ I(x)=U \circ \lim_{n \geq k, \ n \to \infty} I_n(x)=U \circ \lim_{n \geq k, \ n \to \infty} i_n(x)=\lim_{n \geq k, \ n \to \infty} U \circ i_n(x)=$$
	 $$=\lim_{n \geq k, \ n \to \infty} i_n \circ T(x)=\lim_{n \geq k, \ n \to \infty} I_n \circ T(x)=I \circ T(x)$$ 
     and
     $$D_T(x)=\lim_{n \geq k, \ n \to \infty}p_n(x)=\lim_{n \geq k, \ n \to \infty} D_U \circ i_n(x)=D_U \circ \lim_{n \geq k, \ n \to \infty} i_n(x)=D_U \circ I(x).$$
	Since all functions $I,U,T,D_U,D_T$ are continuous, we have
    $$U \circ I(x)=I \circ T(x) \mbox{ and } D_T(x)=D_U \circ I(x)$$ for any $x \in \U.$

    $"(ii) \Rightarrow (iii)"$ - obvious
    
	$"(iii) \Rightarrow (i)"$
	
	We will check that $T$ satisfies condition $(UR)$. Fix $Y\cup \{z\} \subset \U, z\notin Y$ and $1-$Lipschitz retractions $V \colon Y \to Y$ and $V' \colon Y\cup \{z\} \to Y\cup \{z\}$ such that $V'\restriction Y=V$. Fix isometry $J:Y\to \U$ such that 
	$$J \circ V(y)= T\circ J(y) \text{ and } \dist(y,V[Y])=\dist(J(y),T[\U])$$ 
for all $y \in Y$. 
We consider the isometry $j=I\circ J:Y\to \mathbb{U}$. Since
$$j\circ V(y)=(I\circ J) \circ V(y)=I\circ T\circ J(y)=I\circ I^{-1}\circ U\circ I\circ J(y)=U\circ (I\circ J)(y)=U\circ j(y)$$ 
\begin{multline*}\dist(y,V[Y])=\dist(J(y),T[\U])=\dist(J(y),I^{-1}\circ U\circ I[\U])=\dist((I\circ J)(y),U\circ I[\U])\\ =\dist (j(y), U[\U]) \end{multline*}
for $y\in Y$ and the retraction $U$ satisfies (UR) there exists an isometry $j':Y\cup \{z\}\to \mathbb{U}$ extending $j$ such that
$$j '\circ V'(y)=U \circ j'(y) \text{ and } \dist(y,V'[Y\cup \{z\}])=\dist(j'(y),U[\U])$$ for $y \in Y\cup \{z\}$. Let us consider now an isometry $J'=I^{-1}\circ j'$, then
$$J' \circ V'(y)=I^{-1} \circ j' \circ V'(y)=I^{-1} \circ U \circ j'(y)=(I^{-1} \circ U \circ I) \circ I^{-1} \circ j'(y)=T \circ J'(y)$$
	and $$\dist(y,V'[Y\cup\{z\}])=\dist(j'(y),U[\U])=\dist(I^{-1} \circ j'(y),I^{-1} \circ U[\U])=$$
	$$=\dist(J'(y),I^{-1} \circ U[\U])=\dist(J'(y),I^{-1} \circ U \circ I[\U])=\dist(J'(y),T[\U])$$ for any $y\in Y\cup \{z\}$
and $J'$ is an isometry extending $J$, because $j'$ extends $j$.

\end{proof}

\begin{rem} \label{rnr}
If there is an autoisometry $I \colon \U \to \U$ such that 
		$$U \circ I=I \circ T \quad \mbox{ and }  \dist(x,T[\U])=\dist(I(x),U[\U]) \mbox{ for all } x \in \U$$ then $I[T[\U]]=U[\U].$
\end{rem}
\begin{proof}
One can use any of the above equalities. For instance, since $\U$ is a complete metric space and $T[\U],U[\U]$ are closed we have 
$$\dist(x,T[\U])=0 \Leftrightarrow x \in T[\U] \mbox{ and } \dist(I(x),U[\U])=0 \Leftrightarrow I(x) \in U[\U]$$ for any $x \in \U$. Then we have $I[T[\U]] \subset U[\U].$ That inclusion together with surjectivity of $I$ guarantee that $I[T[\U]] = U[\U].$  
\end{proof}

\section{Equivalent variant of an extension property (UR)}

Let us recall a condition
\begin{itemize}
\item[(UR$^{\star}$)] \label{UR2}

for every $\eps >0$, any rational finite triple $(A, r,p)$, every rational isometric embedding $i : A\to (\U_0,U, D )$ such that $$d_{\U_0}(U \circ i(x), i \circ r(x))<\eps \mbox{ and } |D \circ i(x) - p(x)|<\eps,$$ and every rational triple $(B,r',p')$, where $B=A\cup \{b\}, b\notin A$, $r'\restriction A=r$, $p'\restriction A=p$, there exists a rational isometric embedding {$i':B\to (\U_0, U, D)$} such that  $d_{\U}(i(x),i'(x))< 2\eps$ for every $x \in A$ and 
$$d_{\U}(U \circ i'(x), i' \circ r'(x))<5\eps, \  \quad |D \circ i'(x) - p'(x)|<3\eps$$
for every $x \in B.$

\end{itemize}

In this section our aim is to prove that (UR$^{\star}$) is equivalent to (UR).

\begin{lm} \label{rmc}
Let $(X \cup Y,r,p)$ be a finite triple and let $d$ be a metric on $X\cup Y$ such that  $d[X \times X] \subset \Q.$ Then for any $\varepsilon>0$ there are a rational metric $\rho$ on $X \cup Y$ and a rational function $p' \colon (X \cup Y,\rho) \to [0,\infty)$ with $|p'-p|<\varepsilon, \ \rho \restriction X \times X=d \restriction X \times X$ such that $r,p'$ are $1-$Lipschitz with respect to $\rho$ and $d \leq \rho \leq d+\varepsilon.$
\end{lm} 

\begin{proof}
	First step: 
	
	Let us define sets $A$ and $B$ in the following way. If distinct points $x,y,z \in X \cup Y$ satisfy 
	$$d(x,z)=d(x,y)+d(y,z)$$ and at least one distance is irrational then add $(x,y,z)$ to $A$ and add all distances $$d(x,y),d(x,z),d(y,z)$$ to $B$. Set $C=B \setminus \Q$.
	
	We will define a metric $\eta$ on $X \cup Y$ with $d \leq \eta \leq d+\frac{\varepsilon}{4}$ such that $\eta$ and $d$ coincide on $X, \ ((X\cup Y,\eta),r,p)$  is a finite triple and $\eta$ satisfies the strong triangle inequality with respect to the set $A$, i.e. for any $(x,y,z) \in A$ we have $\eta(x,z)<\eta(x,y)+\eta(y,z).$

	Enumerate $C=\{c_1<c_2<\ldots<c_k\}$ for some $k \in \N.$ Set
	$$\varepsilon_1=\min(\{d(x,y)+d(y,z)-d(x,z) \colon x,y,z \in X \cup Y\} \setminus \{0\})$$
	and
	$$\varepsilon_2=\min(\{d(x,y)-d(r(x),r(y)) \colon x,y \in X \cup Y\} \setminus \{0\}).$$
	Put $\varepsilon_0=\frac{\min\{\varepsilon,\varepsilon_1,\varepsilon_2\}}{2}.$
	We define a metric $\eta$ on $X \cup Y$ in the following way: for any $(x,y,z) \in A$ with $c_i=d(u,v)$ for some distinct $u,v \in \{x,y,z\}$ and $1 \leq i \leq k, $ we put $\eta(u,v)=d(u,v)+\frac{\varepsilon_{0}}{2^i}.$ Moreover, if there are $u',v' \in X \cup Y$ such that $r(u')=u, \ r(v')=v$ and $d(u,v)=d(u',v')$ then we put $\eta(u',v')=\eta(u,v)$. For any other $u,v \in X \cup Y$ we put $\eta(u,v)=d(u,v).$
	It is easy to see that $\eta$ is a metric. 
	
	We will check that $\eta$ satisfies the strong triangle inequality with respect to $A$. Suppose that $(x,y,z) \in A$ and distance $d(x,z) \in C$ satisfies $d(x,z)=d(x,y)+d(y,z).$ Then at least one of the numbers $d(x,y),d(y,z)$ lies in $C$, let us say $d(x,y) \in C.$ By the definition of $\eta$, we have $\eta(x,z)=d(x,z)+\frac{\varepsilon_0}{2^{i_1}}$ and $\eta(x,y)=d(x,y)+\frac{\varepsilon_0}{2^{i_2}}$ for some $1 \leq i_2<i_1 \leq k.$ Then we have $$\eta(x,z)=d(x,z)+\frac{\varepsilon_0}{2^{i_1}}< d(x,y)+\frac{\varepsilon_0}{2^{i_2}}+d(y,z)\leq \eta(x,y)+\eta(y,z).$$  Another case $d(x,z)=d(x,y)+d(y,z)$ for $d(x,z) \notin C$ is immediate. Furthermore, if $d(x,z)<d(x,y)+d(y,z)$ then $\eta(x,y)+\eta(y,z)-\eta(x,z) \geq d(x,y)+d(y,z)-d(x,z)-\frac{\varepsilon_1}{2} \geq \frac{\varepsilon_1}{2}>0$ by the definition of $\eta$ and $\varepsilon_0.$
	
	Now, we will see that $r$ and $p$ are $1-$Lipschitz with respect to $\eta.$
	Firstly, see that if $d(x,y)=d(r(x),r(y))$, then we have $\eta(x,y)\leq\eta(r(x),r(y))$ by a definition of $\eta.$ 
	Secondly, if $d(x,y)>d(r(x),r(y))$ and $d(r(x),r(y)) \in C$, then $\eta(x,y)-\eta(r(x),r(y)) \geq d(x,y)-d(r(x),r(y))-\frac{\varepsilon_0}{2^{i}} \geq d(x,y)-d(r(x),r(y))-\frac{\varepsilon_2}{2} \geq \frac{\varepsilon_2}{2}>0$ for some $1 \leq i \leq k.$ Moreover, $p$ is $1-$Lipschitz with respect to $\eta$, since $d \leq \eta.$
	
	Second step:
	
	Set
	$$\varepsilon_3=\min(\{\eta(x,y)+\eta(y,z)-\eta(x,z) \colon x,y,z \in X \cup Y\} \setminus \{0\})$$
	and
	$$\varepsilon_4=\min(\{\eta(x,y)-\eta(r(x),r(y)) \colon x,y \in X \cup Y\} \setminus \{0\}).$$
	
	Put
	$$\varepsilon_5=\min\{\frac{\varepsilon}{4},\varepsilon_3,\varepsilon_4\}.$$ It is easy to see that if we define a metric $\rho$ in such a way that for any $x,y \in X \cup Y$ we get $\eta(x,y) \leq \rho(x,y)<\eta(x,y)+\varepsilon_5$, then the strong triangle inequality with respect to $A$ will hold for $\rho$, too.

	Fix $x,y \in X \cup Y$. If $\eta(x,y) \in \Q$, then we put $\rho(x,y):=\eta(x,y).$ Otherwise, we have two cases. 
	1) If $x,y \in r[X \cup Y]$, then we put any number $\rho(x,y) \in (\eta(x,y),\eta(x,y)+\frac{\varepsilon_5}{2}) \cap \Q$; 2) If $x \notin r[X \cup Y]$ or $y \notin r[X \cup Y]$, then we put any number $\rho(x,y) \in (\eta(x,y)+\frac{\varepsilon_5}{2},\eta(x,y)+\varepsilon_5) \cap \Q.$
	
	It is easy to check that $\rho$ is a metric. By the definition of $\eta$ and $\varepsilon_5$, we have that $d \leq \eta \leq \rho \leq \eta+\frac{\varepsilon}{4} \leq d+\frac{\varepsilon}{2}.$
	
	We will check that $r,p$ are $1-$Lipschitz with respect to $\rho.$ Indeed, we have $\rho(x,y)-\rho(r(x),r(y)) \geq \eta(x,y)-\eta(r(x),r(y))-\varepsilon_5>\varepsilon_4-\varepsilon_4=0$ if at least one of the distances $\rho(x,y),\rho(r(x),r(y))$ lies in $\R \setminus \Q.$ The other case is obvious. Moreover, $p$ is $1-$Lipschitz, since $d \leq \rho.$
	
Third step:

Find $0<a<1$ such that $p'':=a \cdot p$ satisfies $|p-p''|<\frac{\varepsilon}{2}.$ Obviously, $p'' \colon (X \cup Y,\rho) \to [0,\infty)$ is $1-$Lipschitz. Even more, for any distinct $x,y \in X \cup Y$ we have $|p''(x)-p''(y)|<\rho(x,y).$ Now, define $$m=\min\{\rho(x,y)-|p''(x)-p''(y)|\} \setminus \{0\}$$ and $\delta=\frac{\min\{m,\varepsilon\}}{2}.$ For any $x \in X \cup Y$ with $p''(x) \notin \Q$ choose $p'(x) \in (p''(x),p''(x)+\delta) \cap \Q.$ If $p''(x) \in \Q$ put $p'(x)=p''(x).$ Then, for any distinct $x,y \in X \cup Y$ with $p''(x) \geq p''(y)$ we have
$$|p'(x)-p'(y)| \leq \max\{p''(x)+\delta,p''(y)+\delta\}-\min\{p''(x),p''(y)\}=$$
$$ =p''(x)-p''(y)+\delta<|p''(x)-p''(y)|+m \leq |p''(x)-p''(y)|+\rho(x,y)-|p''(x)-p''(y)|=\rho(x,y).$$ Hence, $p' \colon (X \cup Y,\rho) \to [0,\infty) \cap \Q$ is $1-$Lipschitz and $|p'-p|<\varepsilon.$

\end{proof}


\begin{lm} \label{fpb}
Let $\eps > 0, \ A=\{a_1,a_2,\ldots,a_n\}, \ A'=\{a_1',a_2',\ldots,a_n'\}$. Suppose that two metric spaces $(A, d)$, $(A', d')$, {where $d'$ is rational} satisfies $|d(a_i,a_j)-d'(a_{i}',a_{j}')|<\varepsilon$ for any $1 \leq i,j \leq n.$ Suppose there is an isometry $i \colon (A, d) \to \U.$ Then there exists an isometry $i': (A', d') \to \U_0$ such that $d_{\U}(i(a_i), i'(a_i')) <\eps$  for any $1 \leq i \leq n$.
\end{lm}
\begin{proof}
	For any $1 \leq i \leq n$ find $b_i \in \U_0$ with $d_{\U}(b_i,i(a_i))<\frac{\varepsilon}{2}.$ Define a set $B'=\{b_1',b_2',\ldots,b_n'\}$ such that $B' \cap \U=\emptyset.$ Consider a metric $\rho'$ on $B'$ given by $\rho'(b_i',b_j')=d_{\U}(b_i,b_j)$ for any $1 \leq i,j \leq n.$ There is a natural isometry $j' \colon (B',\rho') \to \U_0$ defined by $j'(b_i')=b_i$ for any $1 \leq i \leq n.$ Let us define a metric $\eta$ on $A' \cup B'$ by the formula $$\eta(b_i',a_i')=\eta(a_i',b_i')=\frac{\varepsilon}{2}$$
	$$\eta(b_j',a_i')=\eta(a_i',b_j')=\frac{\varepsilon}{2}+\min\{d'(a_i',a_k')+\rho'(b_k',b_j') \colon 1 \leq k \leq n\} $$
   for any $1 \leq i,j \leq n, \ i \neq j$ and 
	$$\eta\restriction{A' \times A'}=d', \eta\restriction{B' \times B'}=\rho'. $$
	Note that $\eta$ is a rational metric. Then, we use an extension property for $\U_0$ to find an isometry $j \colon (A'\cup B',\eta) \to \U_0$, which extends $j' \colon (B',\rho') \to \U_0.$ Now, consider an isometry $i' \colon (A',d') \to \U_0$ given by $i'=j \restriction{A'}.$ We will check that  $d_{\U}(i(a_i), i'(a_i')) <\eps$  for any $1 \leq i \leq n$. Indeed, we have 
	$$d_{\U}(i(a_i), i'(a_i')) \leq d_{\U}(i(a_i),b_i)+d_{\U}(b_i, i'(a_i'))$$
	$$<\frac{\varepsilon}{2}+d_{\U}(j(b_i'), j(a_i'))=\frac{\varepsilon}{2}+\eta(b_i', a_i')=\frac{\varepsilon}{2}+\frac{\varepsilon}{2}=\varepsilon$$ for any $1 \leq i \leq n.$
\end{proof}

\begin{lm}[\cite{d} {\bf Maximal amalgamation}]\label{maxamal}
	For any finite metric spaces $(X,d_X)$, $(Y,d_Y)$ such that $X\cap Y\neq \emptyset$ and $d_X\restriction (X\cap Y\times  X\cap Y)=d_Y\restriction (X\cap Y\times  X\cap Y)$ a function $	d_{X\cup Y}$ defined by the formula $$d_{X\cup Y}(x,y)=
	\begin{cases}
		d_X(x,y), &  \mbox{ if } x,y\in X,  \\
		d_Y(x,y), &  \mbox{ if } x,y\in Y,  \\
		\min\{d_X(x,z)+d_Y(z,y): z\in X\cap Y\}, & \mbox{ for } x\in X, y\in Y
	\end{cases}$$
	is a metric.
Moreover, if $d_X, d_Y$ are rational, then $d_{X\cup Y}$ is rational, too. If $(X,r_X,p_X),$ $(Y,r_Y,p_Y)$ are finite triples such that  $r_X\restriction X\cap Y= r_Y\restriction X\cap Y$ and $p_X\restriction X\cap Y= p_Y\restriction X\cap Y$, then exist $1-$Lipschitz retraction $r$ and $1-$Lipschitz function $p$ defined on $X\cup Y$.
\end{lm}

\begin{proof}
	The property that the function $d_{X\cup Y}$ is a metric is folklore. 
	Observe that for $x\in X\cap Y$ we have $r_X(x)=r_Y(y)$ and $p_X(x)=p_Y(y)$.
	Let us define $$r(x)=
	\begin{cases}
		r_X(x), &  \mbox{ if } x\in X,  \\
		r_Y(x), &  \mbox{ if } x\in Y,
	\end{cases}\quad \quad p(x)=
	\begin{cases}
		p_X(x), &  \mbox{ if } x\in X,  \\
		p_Y(x), &  \mbox{ if } x\in Y.
	\end{cases}$$ 
	To check that $r$ is a $1-$Lipschitz retraction and $p$ is a $1-$Lipschitz function, we consider only the non-trivial case when $x\in X\backslash (X\cap Y)$ and $y\in Y\backslash (X \cap Y)$. Choose $x_0\in X\cap Y$ such that $d_{X\cup Y}(x,y)=d_X(x,x_0)+d_Y(x_0,y)$, then
	\begin{align*}
		d_{X\cup Y}(r(x),r(y))&\leq d_{X\cup Y}(r(x),r(x_0))+d_{X\cup Y}(r(x_0),r(y))=d_{X}(r_X(x),r_X(x_0))+\\&+d_{Y}(r_Y(x_0),r_Y(y))\leq d_X(x,x_0)+d_Y(x_0,y)=d_{X\cup Y}(x,y)\\
		|p(x)-p(y)|&\leq |p(x)-p(x_0)+p(x_0)-p(y)|\leq  |p(x)-p(x_0)|+|p(x_0)-p(y)|=\\&=|p_X(x)-p_X(x_0)|+|p_Y(x_0)-p_Y(y)|\leq d_X(x,x_0)+d_Y(x_0,y)=\\&=d_{X\cup Y}(x,y).
	\end{align*}
	This completes the proof.
\end{proof}

\begin{lm}[\cite{d} {\bf Minimal amalgamation}]\label{minamal}
	For any finite metric spaces $(X\cup \{a\},d_{X\cup\{a\}})$, $(X\cup \{b\},d_{X\cup\{b\}})$ a function $	d_{X\cup \{a,b\}}$ defined by the formula $$d_{X\cup \{a, b\}}(x,y)=
	\begin{cases}
		d_{X\cup \{a\}}(x,y), &  \mbox{ if } x,y\in X\cup \{a\},  \\
		d_{X\cup \{b\}}(x,y), &  \mbox{ if } x,y\in X\cup \{b\},\\
		\max\{|d_{X\cup\{a\}}(x,z)-d_{X\cup\{b\}}(z,y)|: z\in X\}, & \mbox{ for } x=a, y=b
	\end{cases}$$ 
	is a metric.
	Moreover, if $d_{X\cup \{a\}}, d_{X\cup \{b\}}$ are rational, then $d_{X\cup \{a, b\}}$ is rational, too.
\end{lm}
\begin{lm}\label{newretraction}
Let $\eps>0$, $(A, r,p)$ be a finite triple and  $i : (A, r,p)\to (\U_0, U, D )$ be an isometric embedding  such that $d_{\U_0}(U \circ i(x) , i \circ r(x))<\eps \mbox{ and }  |D \circ i(x) - p(x)|<\eps.$ There are a finite triple $(X,R,P)$  such that $A\subseteq X$, $d_{X}(R(x),r(x))<\eps \mbox{ and } |p(x)-P(x)|<\eps$ for any $x\in A$ and an isometric embedding $I : (X,R,P)\to (\U, U, D )$ such that $I\restriction A=i$ and  $$U \circ I(x) = I \circ R(x) \quad \mbox{ and } D \circ I(x) =P(x),$$  for $x\in X$.

\end{lm}

\begin{proof}
Fix $\eps>0$, a finite triple $(A,r,p)$ and an isometric embedding $i : A\to \U_0$ such that $d_{\U_0}(U \circ i(x) , i \circ r(x))<\eps \mbox{ and }  |D \circ i(x) - p(x)|<\eps$ for any $x\in A$.

Taking $i[A]\subseteq \U$, we define the set $$Y=i[A]\cup U[i[A]]\subseteq \U.$$  Obviously, from the assumption of isometry $i$ we have that $d_\U(i(r(a_l)), U(i(a_l)))<\eps$ and $|D(i(a))-p(a)|\leq \eps$ for any $a\in A$ and $l=1, \dots, m$. 

Let $A=\{a_1, \dots, a_m, r(a_1), \dots, r(a_m)\}$, $C=\{d_1, \dots, d_m, x_1, \dots, x_m\}$ and $X=A\cup C$.
Now we define a function $I:X\to Y$ such that $I(d_l)= U(i(a_l))$, $I(x_l)= U(i(r(a_l)))$ for $l=1, \dots, m$. 

Then we define a metric on $A\cup C$ such that $$d_{A\cup C}(x, y)=d_{\U}(I(x),I(y)) \mbox{ for any } x, y\in C, I(x),I(y)\in  U[i[A]]$$ and  $$d_{A\cup C}(a, y)=d_{\U}(i(a),I(y)) \mbox{ for any } a\in A, y\in C, I(y)\in Y\backslash i[A], i(a)\in i[A].$$ 

Note that $I:X\to Y$ is an isometric embedding; moreover, we define $1-$Lipschitz retraction $R$ as $R(a_l)=R(d_l)=d_l$, $R(r(a_l))=R(x_l)=x_l$ for any $l=1, \dots, m$ and $1-$Lipschitz function $P(x)=\mbox{dist}(I(x),U[\U])$ for any $x\in A\cup C$.  Note that $R[X]=C$ and $d_X(r(a_l),d_l)<\eps$, $d_X(r(a_l),x_l)<\eps$ and $|P(a_l)-p(a_l)|<\eps$, $|P(r(a_l))-p(r(a_l))|<\eps$. Both functions depend on $U$ and $D$, respectively. Observe that $P(d_l)=P(x_l)=0$.

Let us remind that $U$ and $D$ are $1-$Lipschitz. We show that $P$ and $R$ are $1-$Lipschitz. Indeed, for any $x,y\in X$ we have
\begin{align*}
|P(x)-P(y)|&=|D(I(x))-D(I(y))|\leq d_\U(I(x), I(y))\leq d_X(x,y),\\ 
d_X(R(x),R(y))&=d_\U(U(I(x)),U(I(y)))\leq d_\U(I(x),I(y))=d_X(x,y).
\end{align*}
Obviously $$U\circ I(x)= I\circ R(x) \mbox{ and }P(x)=D\circ I(x)$$ for any $x\in X$.
\end{proof}

\begin{lm}\label{aamal}
For any $\eps>0$ and finite triples $(C,R,P)$, $(B, r', p')$ such that $C\cap B=A$ and $B=A\cup\{b\}, b\notin A$, $r'\restriction A=r$ and $p'\restriction A=p$,  $d_C(R(a), r(a)) \leq \eps$ and $|P(a)-p(a)|\leq \eps$ for $a\in A$. Then for any metric $d$ on $C \cup B$ extending both original metrics on $C$ and $B$, there exists a metric $\rho$ on $C\cup B$, retraction $R':C\cup B \to C\cup B$ and function $P':C\cup B\to [0, \infty)$ such that $R'$ and $P'$ become $1-$Lipschitz, $|\rho \restriction (A\cup\{b\})\times(A\cup\{b\})-d\restriction (A\cup\{b\})\times(A\cup\{b\})|~<~2~\eps$, $\rho(R'(x), r'(x))\leq \eps$ and $|P'(x)-p'(x)|\leq \eps$ for any $x\in A\cup\{b\}$.
\end{lm}

\begin{pf}
Let $C, B$ be finite metric spaces $C\cap B=A$ and $B=A\cup\{b\}, b\notin A.$
Fix any metric $d$ defined (for instance, one can use Lemma \ref{maxamal}) on the space $C\cup B$.

We consider two cases:
\begin{itemize}
\item[1.] $r'(b)=b$ and $p'(b)=0$, then we define

$$R'(x)=
\begin{cases}
	r'(b), &  \mbox{ for } x=b,  \\
	R(x), & \mbox{ for } x\in C,
\end{cases}
 \mbox{ and }
 P'(x)=
\begin{cases}
	p'(b), &  \mbox{ for } x=b,  \\
	P(x), & \mbox{ for } x\in C.
\end{cases}$$ 

\item[2.] $r'(b)=a_0$, $a_0\in r[A]$ and $p'(b)\neq 0$, then we define

$$R'(x)=
\begin{cases}
	c_0, &  \mbox{ for } x=b, \mbox{ where } c_0=R(r'(b))=R(r(a_0)), a_0\in A \mbox{ and } R(c_0)=c_0\\
	R(x), & \mbox{ for } x\in C.
\end{cases}$$
 and
 $$P'(x)=
\begin{cases}
	p'(b), &  \mbox{ for } x=b\\
	P(x), & \mbox{ for } x\in C.
\end{cases}$$ 
\end{itemize}

Obviously, in both cases $d(R'(b), r'(b))\leq\eps$ and $|P'(b)-p'(b)|\leq \eps$.

Now we define a metric on $C\cup B$ as $$\rho(x,y)=\max\{d(x,y), d(R'(x),R'(y)), |P'(x)-P'(y)|\}$$ for any $x,y \in C\cup B.$
Then $R'$ and $P'$ become $1-$Lipschitz. 

Moreover, $\rho \restriction A\cup \{b\} < d+2\eps$. We consider only the case when $r'(b)\neq b$. Indeed, fix $a\in B$,
then \begin{align*}d(R'(a),R'(b))&=d(R(a),c_0)<d(R(a),r(a))+d(r(a),r'(b))+d(r'(b),c_0)\\&<2\eps+ d_{A\cup \{b\}}(a,b)
	\end{align*}
and 
\begin{align*}|P'(a)-P'(b)|&=|P(a)-p'(b)|=|P(a)-p(a)+p(a)-p'(b)|\\&<|P(a)-p(a)|+|p(a)-p'(b)|<\eps +d_{A\cup \{b\}}(a,b).	\end{align*}
\end{pf}

\begin{prop}\label{blisko}
Suppose that $X=\bigcup_{i=1}^n X_i$ while $X_i=\{x_i^{(j)} \colon j=1,\ldots,m\}$ for any $i=1,\ldots,n+1.$  Let $\varepsilon>0$ and $((X,d_n),r_n,p_n), \ ((X_{n+1},d),r,p)$ be finite triples such that $|X_i|=|X_{n+1}|$ for $i\in\{1,\dots, n\}$ and $|d_n(x_1^{(j)},x_1^{(k)})-d(x_{n+1}^{(k)},x_{n+1}^{(j)})| \le \eps$ for $j,k\in\{1,\dots, m\}$, $r_n\restriction X_i$ is a retraction for any $i=1,\ldots,n.$ . If $r_n(x_i^{(j)})=x_i^{(k)}$, then $r(x_{n+1}^{(j)})=x_{n+1}^{(k)}$ and $|p_n(x_i^{(j)})-p(x_{n+1}^{(j)}|<\eps$. There is a metric $d_{n+1}$ on $Y=X \cup X_{n+1}$ extending $d_n$ and $d$, a $1-$Lipschitz retraction $r_{n+1}$ defined on $Y$ extending $r_n$ and $r$ and a $1-$Lipschitz function $p_{n+1}$ defined on $Y$ extending both $p_n$ and $p$. Moreover, we can define $d_{n+1}$ in such a way that $d_{n+1}(x_1^{(j)},x_{n+1}^{(j)})=\varepsilon$ for all $1 \leq j \leq m$. 

\end{prop}
\begin{proof}
Using [\cite{P} Example 56.] define metric $d'$ on $X_1\cup X_{n+1}$ as
$$d'(x_{n+1}^{(j)},x_1^{(j)})=d'(x_1^{(j)},x_{n+1}^{(j)})=\eps  \ \textrm{ for any }  \ 1 \leq j \leq m,$$
$$d'(x_i^{(j)},x_{n+1}^{(k)})=d'(x_{n+1}^{(k)},x_{i}^{(j)})=\varepsilon+\min_{1 \leq p \leq m}\{d_n(x_1^{(j)},x_1^{(p)}) +d(x_{n+1}^{(p)},x_{n+1}^{(k)})\}$$
$$ \textrm{for any } 1 \leq j,k \leq m,$$
and $d' \restriction{X_1 \times X_1}=d_n, \ d' \restriction{X_{n+1} \times X_{n+1}}=d.$  Let us define functions $r'_{n+1}$ and $p'_{n+1}$ as follows:
$$r'_{n+1}(x)=
\begin{cases}
	r_n(x), &  \mbox{ for } x\in X_1,  \\
	r(x), & \mbox{ for } x\in X_{n+1}
\end{cases}$$
 and
 $$p'_{n+1}(x)=
\begin{cases}
	p_n(x), &  \mbox{ for } x\in X_1,  \\
	p(x), & \mbox{ for } x\in X_{n+1}.
\end{cases}$$ 
Now we define a metric on $X\cup X_1$ as 
$$\rho(x,y)=\max\{d'(x,y), d'(r'_{n+1}(x),r'_{n+1}(y)), |p'_{n+1}(x)-p'_{n+1}(y)|\}.$$ 
Then the functions $r'_{n+1}$ and $p'_{n+1}$ are $1-$Lipschitz.
Now we use Lemma \ref{maxamal} for the finite triples $((X_1\cup X_{n+1},d'),r'_{n+1},p'_{n+1})$ and $((X,d_n),r_n,p_n)$ to define a metric space $(Y, d_{n+1}).$  Let us define functions $r_{n+1}$ and $p_{n+1}$ as follows:
$$r_{n+1}(x)=
\begin{cases}
	r'_{n+1}(x), &  \mbox{ for } x\in X_1\cup X_{n+1},  \\
	r_n(x), & \mbox{ for } x\in X
\end{cases}$$
 and
 $$p_{n+1}(x)=
\begin{cases}
	p'_{n+1}(x), &  \mbox{ for } x\in X_1\cup X_{n+1},  \\
	p_n(x), & \mbox{ for } x\in X.
\end{cases}$$

\end{proof}




\begin{tw} \label{aprox}
Let $U:\U \to U[\U]$ be a $1-$Lipschitz retraction. The following conditions are equivalent:
\begin{enumerate}
\item[(UR)] \label{UR1}
for every finite triple $(A, r,p)$ and every isometric embedding $i : A\to (\U, U, D )$ such that $U \circ i(x) = i \circ r(x)$ and $D \circ i(x) = p(x)$, and every $B=A\cup \{b\}, b\notin A$, a triple $(B,r',p')$ such that $r'\restriction A=r$, $p'\restriction A=p$, there exists an extension $i':B\to (\U, U, D)$ such that for every $x\in B$ we have $$U \circ i'(x) = i' \circ r'(x)\quad \mbox{ and } \quad D \circ i'(x) = p'(x),$$
\item[(UR$^{\star}$)] \label{UR2}
for every $\eps >0$, any rational finite triple $(A, r,p)$ every rational isometric embedding $i : A\to (\U_0,U, D )$ such that $d_{\U_0}(U \circ i(x), i \circ r(x))<\eps$ and $|D \circ i(x) - p(x)|<\eps$, and every rational triple $(B,r',p')$, where $B=A\cup \{b\}, b\notin A$, $r'\restriction A=r$, $p'\restriction A=p$, there exists a rational isometric embedding {$i':B\to (\U_0, U, D)$} such that for every $x\in A$: {$d_{\U}(i(x),i'(x))< 2\eps$} and for every $x\in B$ we have {$d_{\U}(U \circ i'(x), i' \circ r'(x))<5\eps \quad \mbox{ and } \quad |D \circ i'(x) - p'(x)|<3\eps.$}
\end{enumerate}
\end{tw}

\begin{pf}
{
(UR$^{\star}$) $\Rightarrow$  (UR)
\textbf{Preliminaries:}

Fix a finite triple $((A,d),r, p)$, isometric embedding $i \colon A\to \U$ such that $U \circ i(a) = i \circ r(a)$ and $D \circ i(a) = p(a)$ for any $a\in A$, and a triple $((B,d_B),r',p')$, $B=A\cup \{b\}, b\notin A$, $B=\{x^{(j)} \colon j=1,\ldots,m+1\}$, with $x^{(m+1)}=b$, where  $d_B \restriction A \times A=d,$ $r'\restriction A=r$, $p'\restriction A=p$.

Fix $\eps>0$ and a sequence $\{\eps_n \}_{n\in \omega}$ of decreasing positive numbers such that $21\sum_{n\in \omega} \eps_n<\eps$.
We inductively construct a sequence of rational isometries $\{j'_n\}_{n\in \omega}$ and rational triples $((B_n,d_{n}),r_{n},p_{n})\}_{n\in \omega}$, $B_n=A_n\cup \{b_n\}$, $|A|=|A_n|=m,b_n\notin A_n$ $B_n=\{x_n^{(j)} \colon j=1,\ldots,m+1\}$ with $x_n^{(m+1)}=b_n$ such that:
\begin{enumerate}[(i)]
	\item $|d_n (x_n^{(j)}, x_n^{(j)})-d_B(x^{(j)}, x^{(j)})|<\frac{\eps_n}{2},$ for any $i,j=1,\cdots,m+1$,
	\item $j'_n \colon B_n \to \U_0,$	
	\item  if $r'(x^{(j)})=x^{(i)}$ for some $i,j=1,\cdots,m+1$, then $r'_{n}(x_{n}^{(j)})=x_{n}^{(i)},$ 
	\item  $|p'_{n}(x_{n}^{(j)})-p'(x^{(j)})| \leq \frac{\eps_n}{2}$ for any $j=1,\ldots,m+1$,
	\item $ d_{\U}(j'_n(x_{n}^{(j)}),i(x^{(j)})) <\frac{5\eps_n}{2}$ for any $j=1,\ldots,m,$
	\item $d_\U(U\circ j'_n(x_n^{(j)}),  j'_n\circ r'_n(x_n^{(j)}))<5\eps_n$  and $|D\circ j'_n(x_n^{(j)})-P'_n(x_n^{(j)})|<3\eps_n$, for any $j=1,\cdots,m+1$
	\item $\{j'_n(x^{(j)}_{n})\}_{n\in \omega}$ is Cauchy sequence for any $j=1,\ldots,m+1$.
\end{enumerate} 
First step: Using Lemma \ref{rmc} for space $B$, for $n=0$ we find rational metrics $d_0$ (denote the rational space $B_{0}$) such that $|d_0 (x_0^{(j)}, x_0^{(l)})-d_B(x^{(j)}, x^{(l)})|<\frac{\eps_0}{2}$, $p_{0}'$ is rational and $|p'(x^{(j)})-p'_0(x_0^{(j)})|<\frac{\eps_0}{2}$, $r_{0}' (x_0^{(j)}):=r(x^{(j)})$ is a retraction and both are $1-$Lipschitz. Then $r_{0}:=r'_{0}\restriction A_{0}$ and $p_{0}:=p'_{0}\restriction A_{0}$.	
	
Using Lemma \ref{fpb} we find isometry $j_{0}: (A_{0},d_{0})\to \U_0$ such that 

$d_\U(i(a^{(j)}),j_{0}(a_{0}^{(j)})) <\frac{\eps_0}{2} $. Then
	\begin{align*}d_{\U}(U \circ j_{0}(a_{0}^{(j)}),j_{0} \circ r_{0}(a_{0}^{(j)}))&\leq d_{\U}(U \circ j_{0}(a_{0}^{(j)}),U \circ i(a^{(j)}))+d_{\U}(U \circ i(a^{(j)}), i \circ r(a^{(j)}))\\&+d_{\U}(i \circ r(a^{(j)}),j_{0} \circ r_{0}(a_{0}^{(j)}))<\frac{\eps_0}{2} +\frac{\eps_0}{2}<\eps_0	\end{align*}
	and
	\begin{align*}	|D \circ j_{0}(a_{0}^{(j)}) - p_{0}(a_{0}^{(j)})|&\leq|D \circ j_{0}(a_{0}^{(j)}) -D \circ i(a^{(j)})| +|D \circ i(a^{(j)})-p(a^{(j)})|\\&+|p(a^{(j)})-p_{0}(a_{0}^{(j)})|<d_\U(j_{0}(a_{0}^{(j)}),i(a^{(j)}))+\frac{\eps_0}{2} <\eps_0
	\end{align*}
	for $a^{(j)}\in A, a_{0}^{(j)}\in A_{0}$.

Using condition (UR$^{\star}$) we get an isometry $j'_0 \colon (B_0, d_0) \to \U_0$ such that $j'_0$ is a $5\eps_0$ commuting with $U$ and $3\eps_0$ commuting with $D$ and $d_{\U_0}(j_0(a_{0}^{(j)}), j'_0(a_{0}^{(j)}))<2\eps_0$, especially $d_{\U_0}(i(a^{(j)}), j'_0(a_{0}^{(j)}))<\frac{5\eps_0}{2}$, for $a^{(j)}\in A, a_{0}^{(j)}\in A_{0}$

Using Proposition \ref {blisko} for spaces $B$ and $B_0$ we get metric $\rho'_0$ defined on $B\cup B_0$ such that $\rho'_0(x^{(k)}, x_0^{(k)})<\frac{\eps_0}{2}$, $x^{(k)}\in B, x_0^{(k)}\in B_0$,  $1-$Lipschitz retraction $\bar{R}_0$ and $1-$Lipschitz function $\bar{P}_0$ such that $\bar{R}_0\restriction B=r'$, $\bar{R}_0\restriction B_0=r'_n$ and $\bar{P}_0\restriction B=p'$, $\bar{P}_0\restriction B_0=p'_0$.

\textbf{Construction of $\rho'$ and $j'$:}

Induction: Suppose we defined a metric $\rho_n'$  on the set  the set $B\cup \bigcup_{k=0}^{n} B_k$  such that $\rho_n'(x^{(j)},x_{n}^{(j)})<\frac{\eps_n}{2}$ for $x^{(j)} \in B$, $x_{n}^{(j)}\in B_n$, $j=1, \dots, m+1$, an isometry $j_n' \colon (B_n, d_n) \to \U_0$ such that $d_\U(U\circ j'_n(x_n^{(j)}),  j'_n\circ r'_n(x_n^{(j)}))<5\eps_n$, $|D\circ j'_n(x_n^{(j)})-p'_n(x_n^{(j)})|<3\eps_n$, $j=1,\cdots,m+1$ and $ d_{\U}(j'_n(a_{n}^{(j)}),i(a^{(j)})) \leq \frac{5 \eps_n}{2}$ for $a^{(j)}\in A, a_n^{(j)}\in A_n$, $j=1, \dots, m$.

Using Lemma \ref{newretraction} for the triple $(B_n, r'_n, p'_n)$ and the isometry $j_n':  (B_n, d_n) \to \U_0$  we obtain commutative isometry $j''_n:  (B_n\cup C_{n},d'_n)\to (j_n'[B_n]\cup U[j_{n}'[B_n]], d_\U)$ and $1-$Lipschitz retraction $r''_n$ and $1-$Lipschitz function $p''_n$ such that 
$d'_n~(r'_n(x_n^{(j)})~, r''_n(x_n^{(j)}))~<~5\eps_n$ and $|p'_n(x_n^{(j)})-p''(x_n^{(j)})|<3\eps_n$ for any $x_n^{(j)}\in B_n$, $d'_n\restriction B_n^2 =d_n$, $j=1,\dots, m+1$ (it may happen that $j'_n$ commutes with $U$ and $D$, then $C_n=\emptyset$ and all considerations are simpler.)

Taking into account commutative isometry $i:(A,r,p)\to (\U,U,D)$ we get the metric $f_n$ on the set $A \cup B_n\cup C_{n}$  and commutative isometry $I_n: (A \cup B_n\cup C_{n},f_n)\to i[A]\cup  j_n'[B_n]\cup U[j_{n}'[B_n]]$, and $1-$Lipschitz retraction $R_n$ and $1-$Lipschitz function $P_n$ determined by $U$ and $D$, such that $P_n\restriction (B_n\cup C_n)=p''_n$, $R_n\restriction (B_n\cup C_n)=r''_n$, $P_n\restriction A=p$ and $R_n\restriction A=r$, $f_n(a^{(j)},a_n^{(j)})<\frac{5\eps_n}{2}$ and again $f_n\restriction B_n^2 =d_n$. Denote $e_n=f_n \restriction (A \cup A_n \cup C_n) \times (A \cup A_n \cup C_n).$

Using Lemma \ref{maxamal} for $(B, d_B)$ and $((A \cup A_n \cup C_n),e_n)$ we define metric $g_n$ on $B \cup A_n\cup C_n$ with suitable $1-$Lipschitz retraction $R'_n$ and $1-$Lipschitz function $P'_n$ such that $P'_n\restriction (A \cup A_{n}\cup C_{n})=P_n\restriction (A \cup A_{n}\cup C_{n})$, $R'_n\restriction (A\cup A_n\cup C_n)=R_n\restriction (A \cup A_{n}\cup C_{n})$, $P'_n\restriction B=p'$ and $R'_n\restriction B=r'$.

Using Lemma \ref{minamal} for $(B \cup A_n\cup  C_n,g_n)$ and $(A\cup B_n\cup  C_n,f_n)$ we define metric $g_n'$ on $B \cup B_n\cup C_n$ extending $g_n$ and $f_n$, respectively, we still have the same as before retraction $\bar{r}$ and function $\bar{p}$, extending $R'_n$, $R_n$ and $P'_n, P_n$

\begin{claim} $g_n'(b,b_n)<\frac{25\eps_n}{2}$.
\end{claim}

\begin{proof}
Let us remain that $g_n'(b,b_n)=\max\{|g_n(b,x)-f_n(x,b_n)|:x\in A\cup A_n \cup C_n\}$.
Let $x^{(j)}$ be such a point that $g_n'(b,b_n)=|g_n(b,x^{(j)})-f_n(x^{(j)},b_n)|$.
We consider three cases:
\begin{itemize}
\item[1)] $x=x^{(j)}\in A$, then $g_n(b,x^{(j)})=d_B(b,x^{(j)})$ and $f_n(x_n^{(j)}, x^{(j)})\leq \frac{5\eps_n}{2}$, then $f_n(x^{(j)},b_n)\leq f_n(x^{(j)}, x_n^{(j)})+ f_n(x_n^{(j)},b_n)\leq \frac{5\eps_n}{2}+d_n(x_n^{(j)},b_n)$ and $d_n(x_n^{(j)},b_n)= f_n(x_n^{(j)},b_n)\leq f_n(x_n^{(j)}, x^{(j)})+ f_n(x^{(j)},b_n) \leq \frac{5\eps_n}{2}+f_n(x^{(j)},b_n)$, this implies that $d_n(x_n^{(j)},b_n)-\frac{5\eps_n}{2}\leq f_n(x^{(j)},b_n)\leq\frac{5\eps_n}{2}+d_n(x_n^{(j)},b_n)$. Then

$|g_n(b,x^{(j)})-f_n(x^{(j)},b_n)|<3 \epsilon_n$
since $$g_n(b,x^{(j)})-f_n(x^{(j)},b_n)<d_B(b,x^{(j)})-d_n(x_n^{(j)},b_n)+\frac{5\eps_n}{2}<\frac{5\eps_n}{2}$$ and 
$$f_n(x^{(j)},b_n)-g_n(b,x^{(j)})\leq d_n(x_n^{(j)},b_n)+\frac{5\eps_n}{2}-d_B(b,x^{(j)})<3\eps_n.$$
\item[2)] $x=x_n^{(j)}\in A_n$, then $f_n(b_n,x_n^{(j)})=d_n(b_n,x_n^{(j)})$ and $g_n(x_n^{(j)}, x^{(j)})=f_n(x_n^{(j)}, x^{(j)})\leq \frac{5\eps_n}{2}$, the we get $d_B(x^{(j)},b)-\frac{5\eps_n}{2}\leq g_n(x^{(j)},b)\leq\frac{5\eps_n}{2}+d_B(x^{(j)},b)$. Then

$|g_n(b,x^{(j)})-f_n(x^{(j)},b_n)|=\begin{cases} g_n(b,x^{(j)})-f_n(x^{(j)},b_n)\leq\frac{5\eps_n}{2}\\
f_n(x^{(j)},b_n)-g_n(b,x^{(j)})\leq 3\eps_n.\end{cases}$
\item[3)] $x\in C_n$, $x=I_n^{-1}(U(j_n'(x_n^{(j)})))$, observe that $$f_n(r'_n(x_n^{(j)}),x)=f_n(r'_n(x_n^{(j)}),I_n^{-1}(U(j_n'(x_n^{(j)}))))=d_\U(I_n(r'_n(x_n^{(j)})),U(j_n'(x_n^{(j)})))=$$ $$=d_\U(j''_n(r'_n(x_n^{(j)})),U(j''_n(x_n^{(j)})))=d_\U(j''_n(r'_n(x_n^{(j)})),j''_n(r''_n(x_n^{(j)})))=
$$ $$=d'_n(r'_n(x_n^{(j)}),r''_n(x_n^{(j)}))<5\eps_n$$ and $g_n(r'_n(x_n^{(j)}), x)=g_n((r'_n(x_n^{(j)}),I_n^{-1}(U(j_n'(x_n^{(j)}))))=d_\U(I_n(r'_n(x_n^{(j)})),U(j_n'(x_n^{(j)})))<5\eps_n$,

then $d_n(b_n,r'_n(x_n^{(j)}))-5\eps_n\leq f_n(b_n,x)\leq d_n(b_n,r'_n(x_n^{(j)}))+5\eps_n$, where $r'_n(x_n^{(j)})\in A_n$, moreover $f_n(x_n^{(j)}, x^{(j)})\leq \frac{5\eps_n}{2}$ and $d_B(r(x^{(j)}),b)-\frac{5\eps_n}{2}-5\eps_n\leq g_n(x,b)\leq\frac{5\eps_n}{2}+5\eps_n+d_B(r(x^{(j)}),b)$. We show only one inequality, the second one is similar:
$d_B(r(x^{(j)}),b)=g_n(r(x^{(j)}),b)\leq g_n(r(x^{(j)}),r_n(x_n^{(j)}))+g_n(r_n(x_n^{(j)}),x)+g_n(x,b)\leq\frac{5\eps_n}{2}+5\eps_n+g_n(x,b)$. Then

$|g_n(b,x)-f_n(x,b_n)|= \begin{cases}g_n(b,x)-f_n(x,b_n)\leq \frac{25\eps_n}{2}\\
f_n(x,b_n)-g_n(b,x)\leq \frac{25\eps_n}{2}.\end{cases}$
\end{itemize}
\end{proof}

\begin{claim} \label{cl1}
There exists a metric $\rho_n$ on $B \cup B_n\cup C_n$ such that:
\begin{itemize}
\item[1)] $\rho_n \restriction [A \cup B_n\cup C_n]^2=g'_n\restriction [A\cup B_n\cup C_n]^2=f_n$,
\item[2)]  $\rho_n \restriction [B \cup A_n\cup C_n]^2=g'_n\restriction [B \cup A_n\cup C_n]^2=g_n$,
\item[3)] $\bar{r}:B \cup B_n\cup C_n\to B \cup B_n\cup C_n$ define as $\bar{r}(x)=
\begin{cases}
	R'_n(x), &  \mbox{ if } x\in B\cup A_n\cup C_n,  \\
	R_n(x), &  \mbox{ if } x\in A\cup B_n\cup C_n,
\end{cases}$ is a $1-$Lipschitz retraction,
\item[4)] $\bar{p}:B \cup B_n\cup C_n\to [0,\infty)$ define as $\bar{p}(x)=
\begin{cases}
	P'_n(x), &  \mbox{ if } x\in B\cup A_n\cup C_n,  \\
	P_n(x), &  \mbox{ if } x\in A\cup B_n\cup C_n,
\end{cases}$ is a $1-$Lipschitz function,
\item[5)]  $\rho_n(b, b_n)<\frac{35\eps_n}{2}$.
\end{itemize}
\end{claim}

\begin{proof}
Let us define $$\rho_n(x, y)=\max\{g'_n(x,y), g'_n(\bar{r}(x),\bar{r}(y)), |\bar{p}(x)-\bar{p}(y)|\}$$
for any $x,y\in B \cup B_n\cup C_n$.
Conditions 1)--4) are obvious. We have to check only $\rho_n(b, b_n)<\frac{35\eps_n}{2}$ in 5).
Let us consider three cases:
\begin{itemize}
\item[1)] $\max\{g'_n(b,b_n), g'_n(\bar{r}(b),\bar{r}(b_n)), |\bar{p}(b)-\bar{p}(b_n)|\}=g'_n(b,b_n)$, then $\rho_n(b,b_n)<\frac{25}{2}\eps_n$
\item[2)] $\max\{g'_n(b,b_n), g'_n(\bar{r}(b),\bar{r}(b_n)), |\bar{p}(b)-\bar{p}(b_n)|\}=g'_n(\bar{r}(b),\bar{r}(b_n))$, then \\
$\rho_n(b,b_n)=g'_n(\bar{r}(b),\bar{r}(b_n))=g'_n(r'(b),R_n(b_n))=g'_n(r'(b),r''_n(b_n))\leq g'_n(r'(b),r'_n(b_n))+ d'_n(r'_n(b_n),r''_n(b_n))<\frac{25}{2}\eps_n+5\eps_n=\frac{35}{2}\eps_n$, if $\bar{r}(b)=b, \bar{r}(b_n)=r''_n(b_n)$ \\
otherwise $\rho_n(b,b_n)=g'_n(\bar{r}(b),\bar{r}(b_n))=g'_n(r'(b),R_n(b_n))\leq f_n(r(a^{(k)}),r'_n(a_n^{(k)}))+d'_n(r'_n(a_n^{(k)}),r''_n(a_n^{(k)}))<\frac{5\eps_n}{2}+5\eps_n<8\eps_n$, for $\bar{r}(b)=r(a^{(k)})=a^{(k)}$, $\bar{r}(b_n)=r''_n(a_n^{(k)})$,
\item[3)] $\max\{g'_n(b,b_n), g'_n(\bar{r}(b),\bar{r}(b_n)), |\bar{p}(b)-\bar{p}(b_n)|\}=|\bar{p}(b)-\bar{p}(b_n)|$, then $\rho_n(b,b_n)=|\bar{p}(b)-\bar{p}(b_n)|=0$ if $\bar{p}(b)=0$ and $\bar{p}(b_n)=0$,\\otherwise
$\rho_n(b,b_n)=|\bar{p}(b)-\bar{p}(b_n)|=|p'(b)-P_n(b_n)|=|p'(b)-p'_n(b_n)+p'_n(b_n)-p''_n(b_n)|\leq |p'(b)-p'_n(b_n)|+|p'_n(b_n)-p''_n(b_n)|<\frac{\eps_n}{2}+3\eps_n<4\eps_n$
\end{itemize}

\end{proof}

Enumerate sets $B \cup B_n\cup C_n=\{x^{(j)} \colon j=1,\ldots,2m+k+2\}$ and $B_{n+1} \cup B_n\cup D_n=\{x_{n+1}^{(j)} \colon j=1,\ldots,2m+k+2\}$ with $$x^{(j)}=\begin{cases}a^{(j)} \textrm{ for } j=1,\ldots,m, \\ a_n^{(j-m)} \textrm{ for } j=m+1,\ldots,2m, \\ b_n \textrm{ for } j=2m+1, \\ c_n^{(j-2m-1)} \textrm{ for } j=2m+2,\ldots,2m+k+1 \\ b \textrm{ for } j=2m+k+2, \end{cases}$$
and
$$x_{n+1}^{(j)}=\begin{cases}a_{n+1}^{(j)} \textrm{ for } j=1,\ldots,m, \\ a_n^{(j-m)} \textrm{ for } j=m+1,\ldots,2m, \\ b_n \textrm{ for } j=2m+1, \\ d_n^{(j-2m-1)} \textrm{ for } j=2m+2,\ldots,2m+k+1, \\ b_{n+1} \textrm{ for } j=2m+k+2, \end{cases}$$
where $A_n=\{a_n^{(j)} \colon j=1,\ldots,m\}, \  C_n=\{c_n^{(j)} \colon j=1,\ldots,k\}, \ D_n=\{d_n^{(j)} \colon j=1,\ldots,k\},$ note that there is $k \in \N$ such that for all $n \in \N$ we have $|A_n|=|A_{n+1}|=|A|=m \geq k=|D_n|=|C_n|.$ 

Using Claim \ref{cl1} we get a finite triple $((B \cup B_n\cup C_n,\rho_{n}),\bar{r},\bar{p})$. Applying Lemma \ref{rmc} for above triple we get a rational triple $((B_{n+1} \cup B_n\cup D_n,d_{n+1}),r'_{n+1},p'_{n+1})$, such that $|d_{n+1}(x_{n+1}^{(j)}, x_{n+1}^{(l)})-\rho_n(x^{(j)}, x^{(l)})|<\frac{\eps_{n+1}}{2}$, $r'_{n+1}(x_{n+1}^{(j)})=\bar{r}(x^{(j)})$ and $|p'_{n+1}(x_{n+1}^{(j)})-\bar{p}(x^{(j)})|<\frac{\eps_{n+1}}{2}$, $x_{n+1}^{(j)}, x_{n+1}^{(l)}\in B_{n+1} \cup B_n\cup D_n$, $x^{(j)}, x^{(l)}\in B\cup B_n \cup C_n$ for $j, l=1, \dots, 2m+k+2$ . Note that $d_{n+1}\restriction B_n \times B_n=d_n$ and \begin{align*}&|d_{n+1}(x_{n+1}^{(l)}, x_{n+1}^{(j)})-g_n(x^{(l)}, x^{(j)})|=|d_{n+1}(x_{n+1}^{(l)}, x_{n+1}^{(j)})-\rho_n(x^{(l)}, x^{(j)})|<\frac{\eps_{n+1}}{2},
	\end{align*}
for $x_{n+1}^{(l)}, x_{n+1}^{(j)}\in B_n, x^{(l)}, x^{(j)}\in B$ and
\begin{align*}&d_{n+1}(x_{n+1}^{(l)}, x_{n}^{(l)})\leq \rho_n(x^{(l)}, x_n^{(l)})+\frac{\eps_{n+1}}{2}\leq 18\eps_{n+1}\end{align*} for $x_{n+1}^{(l)}\in B_{n+1}, x_{n}^{(l)}\in B_n, x^{(l)}\in B$.
 Using Lemma \ref{fpb} for $(A \cup B_n\cup C_n,f_n)$ and isometry $I_n$ we find isometry $h_{n+1}:(A_{n+1} \cup B_n\cup D_n, d_{n+1}\restriction[A_{n+1} \cup B_n\cup D_n]^2 )\to \U_0$ and such that  $d_\U(I_n(x^{(j)}), h_{n+1}(x_{n+1}^{(j)})) <\frac{\eps_{n+1}}{2}$
	\begin{align*}d_{\U}(U \circ h_{n+1}(x_{n+1}^{(j)}),h_{n+1} \circ r'_{n+1}(x_{n+1}^{(j)}))<\eps_{n+1}, \,|D \circ h_{n+1}(x_{n+1}^{(j)}) - p'_{n+1}(x_{n+1}^{(j)})|<\frac{\eps_{n+1}}{2}
	\end{align*}
	for $x^{(j)}\in A\cup B_n\cup C_n, x_{n+1}^{(j)}\in A_{n+1} \cup B_n\cup D_n$.
	
Using condition (UR$^{\star}$)  we get an isometry $i_{n+1} \colon (B_{n+1} \cup B_n\cup D_n, d_{n+1}) \to \U_0$ such that $i_{n+1}$ is a $5\eps_{n+1}$ commuting with $U$ and $3\eps_{n+1}$ commuting with $D$. 

Moreover $d_{\U}(i_{n+1}(x_n^{(j)}), I_n(x^{(j)}))<\frac{5\eps_{n+1}}{2}$,$x_{n+1}^{(j)}, x_{n+1}^{(l)}\in A_{n+1}\cup B_n\cup D_n$, $x^{(j)}, x^{(l)}\in A \cup B_n\cup C_n$ for $j, l=1, \dots, 2m+k+1$.

Applying Proposition \ref{blisko} we extend metric $\rho'_n$ to metric $\rho'_{n+1}$ on space $B \cup \bigcup_{k=0}^{n+1} B_k$  such that $\rho_{n+1}'(x^{(j)}, x_{n+1}^{(j)})<\frac{\eps_{n+1}}{2}$, for $x^{(j)}\in B, x_n^{(j)}\in B_n$, $j=1,\dots, m+1$ and $1-$Lipschitz function to $\bar{R}_{n+1}, \bar{P}_{n+1}$ such that $\bar{R}_{n+1}\restriction (B\cup \bigcup_{k=0}^{n} B_k)=\bar{R}_{n}$, $\bar{R}_{n+1}\restriction B_{n+1}=r'_{n+1}$ and $\bar{P}_{n+1}\restriction (B\cup \bigcup_{k=0}^{n} B_k)=\bar{P}_{n}$,  $\bar{P}_{n+1}\restriction B_{n+1}=p'_{n+1}$.

Denote $j'_{n+1}:=i_{n+1}\restriction B_{n+1}$, it is an isometry.

Let us define a metric $\rho'=\bigcup_{n=1}^{\infty} \rho_n'$ on the set $B \cup \bigcup_{n=1}^{\infty} B_n$ and a function $j'=\bigcup_{n=1}^{\infty} j_{n}' \colon (\bigcup_{n=1}^{\infty} B_n,\rho') \to \U_0 \subset \U.$ Then $j' \restriction B_n=i_{n} \restriction B_n=j'_n \colon B_n\to \U_0$ is an isometry. Moreover $d_{\U}(i_{n+1}(x_n^{(j)}), I_n(x_{n}^{(j)}))=d_{\U}(i_{n+1}(x_n^{(j)}), j'_n(x_{n}^{(j)}))=d_{\U}(i_{n+1}(x_n^{(j)}), i_n(x_{n}^{(j)}))<\frac{5\eps_{n+1}}{2}$ for $x_n\in B_n$, $j=1,\dots, m+1$.

\textbf{Construction of $j^{\star}$ and $i^{\star}$:}

Now, let us see that $$\cl_{B \cup \bigcup_{n=1}^{\infty} B_n}( \bigcup_{n=1}^{\infty} B_n)=B \cup \bigcup_{n=1}^{\infty} B_n$$ with respect to the metric $\rho'.$ 
Indeed, a repeated construction of $\rho_n$ guarantees that $\rho'(x_n^{(k)},x^{(k)}) \leq \frac{\eps_{n}}{2}$ for any $x^{(j)}\in B$, $x_n^{(j)}\in B_n$, consequently, $\lim_{n \to \infty} x_{n}^{(k)}=x^{(k)}$.

Let us consider an extension $j^{\star} \colon B\cup \bigcup_{n=1}^{\infty} B_n \to \U$ of $j'$ given by 
$$j^{\star}(x)=\begin{cases} j'(x) & \textrm{ if } x \in \bigcup_{n=1}^{\infty} B_n\\ \lim_{n \to \infty} j'(x_n^{(k)}) & \textrm{ if } \lim_{n \to \infty} x_n^{(k)}=x \in B \ \textrm{ for some } x_n^{(k)}\in B_n. \end{cases}$$ 

In order to verify that such a function is well-defined we have to argue for the existence of all limits $\lim_{n \to \infty} j'(x_n^{(k)}), \ k=1,\ldots,m+1.$ Indeed, such limit exists, since for any $n,l \in \N$ with $n>l$ and $k=1,\ldots,m+1$ we have
\begin{align*}
    21\sum_{j=l}^{n-1}\eps_j&>\sum_{j=l}^{n-1}(d_{j+1}(x_{j+1}^{(k)},x_{j}^{(k)})+d_{\U}(i_{j+1}(x_j^{(k)}),i_{j}(x_{j}^{(k)})))=\\& =\sum_{j=l}^{n-1} (d_{\U}(i_{j+1}(x_{j+1}^{(k)}),i_{j+1}(x_{j}^{(k)}))+d_{\U}(i_{j+1}(x_j^{(k)}),i_{j}(x_{j}^{(k)}))) \\& 
    \geq d_{\U}(j'_n(x_n^{(k)}),j'_l(x_l^{(k)}))=d_{\U}(j'(x_n^{(k)}),j'(x_l^{(k)})),
\end{align*}
 which shows that $(j'(x_n^{(k)}))_{n \in \omega}$ is a Cauchy sequence in a complete space $\U$ for any $j,l=1,\ldots,m+1.$

Moreover such formula guarantees a continuity of $j^{\star}.$ We define $i^{\star}=j^{\star} \restriction B \colon B \to \U.$ 

\textbf{We will check that $i^{\star} \restriction A=i.$} Fix $k=1,\ldots,m.$ We have
\begin{align*}&d_{\U}(i^{\star}(a^{(k)}),i(a^{(k)}))=d_{\U}(j^{\star}(a^{(k)}),i(a^{(k)}))=\lim_{n \to \infty} d_{\U}(j^{\star}(a_n^{(k)}),i(a^{(k)}))\\&=\lim_{n \to \infty} d_{\U}(j'(a_n^{(k)}),i(a^{(k)}))=\lim_{n \to \infty} ( d_{\U}(j'(a_n^{(k)}),j'_n(a_n^{(k)}))\\&+d_{\U}(j'_n(a_n^{(k)}),i(a^{(k)}))) \leq \lim_{n \to \infty}\frac{5\eps_n}{2}\to 0\end{align*}
 as $n\to \infty$, then we get $i^{\star}(a^{(k)})=i(a^{(k)})$ for any $k=1,\ldots,m.$

Furthermore, \textbf{$i^{\star}$ is an isometry}. Indeed, fix $k=1,\ldots,m$ and note that 
\begin{align*}&d_{\U}(i^{\star}(a^{(k)}),i^{\star}(b))=d_{\U}(j^{\star}(a^{(k)}),j^{\star}(b))=\lim_{n \to \infty} d_{\U}(j^{\star}(a_{n}^{(k)}),j^{\star}(b_n))=\\&\lim_{n \to \infty} d_{\U}(j'(a_{n}^{(k)}),j'(b_n))=\lim_{n \to \infty} d_{\U}(j_n'(a_{n}^{(k)}),j_n'(b_n))=
\lim_{n \to \infty} \rho_n'(a_{n}^{(k)},b_n)\\&=\lim_{n \to \infty} \rho'(a_{n}^{(k)},b_n)=\rho'(a^{(k)},b).
\end{align*}

Now, we will check that $i^{\star}$ \textbf{is commuting} with respect to $U$.
We have
\begin{align*}&d_{\U}(U(i^{\star}(b)),i^{\star}(r'(b)))\leq d_{\U}(U(i^{\star}(b)),U(j_n'(b_n)))+d_{\U}(U(j_n'(b_n)),j_n'(r_n(b_n)))
\\&+d_{\U}(j_n'(r_n(b_n)),i^{\star}(r'(b))) \leq d_{\U}(i^{\star}(b),j_n'(b_n))+5\eps_n
+d_{\U}(j_n'(r_n(b_n)),i^{\star}(r'(b)))\\& \leq d_{\U}(j^{\star}(b),j^{\star}(b_n))+5\eps_n
 +d_{\U}(j^{\star}(r_n(b_n)),j^{\star}(r'(b)))<5\eps_n\to 0\end{align*}
 when $n\to \infty$ by a continuity of $j^{\star}$ and condition $(i)$.

Similarly,
$$ |D \circ j^{\star}(b)-p'(b)|=\lim_{n \to \infty}|D \circ j_n'(b_n)-p_n(b_n)| \leq \lim_{n \to \infty} 3\eps_n\to 0 \mbox{ when } n\to \infty.$$

}

(UR) $\Rightarrow$  (UR$^{\star}$)
Fix a rational finite triple $((A,d_A),r,p)$, isometric embedding $i : A\to \U_0$ such that $d_{\U_0}(U \circ i(x), i \circ r(x))<\frac{1}{n}$ and $|D \circ i(x) - p(x)|<\frac{1}{n}$, and a rational triple $(B,r',p')$, $B=A\cup \{b\}, b\notin A$ equipped with a metric $d_B$, $r'\restriction A=r$, $p'\restriction A=p$.

Using Lemma \ref{newretraction} we find a metric space $(C, d_{C})$, $A\subseteq C$, with $1-$Lipschitz retraction $R$ and $1-$Lipschitz function $P$ such that  $d_{C}(R(x),r(x))\leq \frac{1}{n}$, $|P(x)-p(x)|\leq \frac{1}{n}$ for $x\in A$ and isometry $I:C\to \U$ such that $U \circ I(x) = I \circ R(x)$,  $D \circ I(x) = P(x)$ for any $x\in C$ and $I\restriction A=i$.

Let $d_{\{b\}\cup C}$ denote a metric obtained by Lemma \ref{maxamal} for spaces $C$ and $B$, where $C\cap B =A$. Applying Lemma \ref{aamal} we find a metric $\rho$ on $\{b\}\cup C$ such that $\rho\restriction B\times B-d_B<\frac{2}{n}$, $1-$Lipschitz functions $R'$, $P'$ such that $\rho(R'(x),r'(x))\leq \frac{1}{n}$, $|P'(x)-p'(x)|\leq \frac{1}{n}$ for $x\in B$.

Condition (UR) give us an isometry $I': (\{b\}\cup C,\rho)\to \U$ such that  $I'\restriction C= I$, $I'\restriction A=i$ and $x\in \{b\}\cup C$ we have $$U \circ I'(x) = I' \circ R'(x)\quad \mbox{ and } \quad D \circ I'(x) = P'(x).$$

Then, Lemma \ref{fpb} guarantee existence of an isometry $j':(\{b\}\cup A, d_{\{b\}\cup A})\to \U_0$ such that $d_\U(j'(x),I'(x)) <\frac{2}{n}$ for any $x\in \{b\}\cup A$.

Let us consider the isometry $i':=j'\restriction B$, then $i':B\to \U_0$ is such that $d_\U(i'(x),I'(x)) <\frac{2}{n}$ for $x\in B$, especially $d_\U(i'(x),i(x)) <\frac{2}{n}$ for any $x\in A$.

Moreover for any $x\in B$, $a\in A$ we have
\begin{align*}&d_{\U}(U \circ i'(x), i' \circ r'(x))=d_{\U}(U \circ i'(x),U \circ I'(x))+d_{\U}(U \circ I'(x), I' \circ R'(x))\\&+d_{\U}(I' \circ R'(x), I'\circ r'(x))+d_{\U}(I' \circ r'(x),i' \circ r'(x))<\frac{2}{n} +\frac{1}{n}+\frac{2}{n}=\frac{5}{n}\\
 &|D \circ i'(x) - p'(x)|=|D \circ i'(x) -D \circ I'(x)|+|D \circ I'(x)-P'(x)|+|P'(x)- p'(x)|\\&<d_\U(i'(x),I'(x))+\frac{1}{n}<\frac{3}{n}.
\end{align*}
\end{pf}

\section{A space of retractions}



Let us recall the notations
\begin{align*}
\mathcal{R}(\U)=&\{ R:\U\to \U \colon  R \text{ is $1-$Lipschitz retraction}\}
\end{align*}
\mbox{ and}
\begin{align*}
\mathcal{U}(\U)=&\{ U \in \mathcal{R}(\U) \colon  U \text{ is universal, ultrahomogeneous}\}.
\end{align*}
We study the set $\mathcal{R}(\U)$ and its subspaces with three topologies of: pointwise convergence $\tau_p$, pointwise retract convergence $\tau_{pr}$ and uniform convergence $\tau_u$. Below we describe basic open neighbourhoods in both topologies $\tau_p,\tau_{u}$.

The basic open neighbourhoods of $1-$Lipschitz retraction $U$ endowed with the pointwise convergence topology $\tau_p$ are of the form
$$\mathcal{O}(U)_{X,\eps}=\{R\in \mathcal{R}(\U): d_\U(R(x_i),U(x_i))<\eps, x_1,\dots x_n\in X\},$$
with $X=\{x_1,\ldots,x_n\} \subset \U$ is finite and $\varepsilon>0$.

Respectively, basic open neighbourhoods of $1-$Lipschitz retraction $U$ in  the uniform convergence topology $\tau_u$ are of the form
$$V(U)_{\varepsilon}=\{R\in \mathcal{R}(\U): d_\U(R(x),U(x))<\eps \textrm{ for any } x\in \U\}.$$

 Now, we will check that $\tau_{pr}$ is a topology. More precisely, we will verify that for any $U \in \mathcal{R}(\U)$ the family $$\{W(U)_{X,\eps} \colon X=\{x_1,x_2,\ldots,x_n\} \subset \U, \ \eps>0\}$$ is a base at a point $U$. Indeed, fix $U \in \mathcal{R}(\U), X=\{x_1,x_2,\ldots,x_n\} \subset \U, \ \eps>0$ and $R \in W(U)_{X,\eps}$. Since $X$ is finite, there is $\eta>0$ with 
$$\forall_{i=1,\ldots,n} \ (d_{\U}(R(x_i),z)<\eta \mbox{ and } |D_{R}(x_i)-y|<\eta) \Rightarrow (d_{\U}(U(x_i),z)<\eps \mbox{ and } |D_{U}(x_i)-y|<\eps).$$ Then $W(R)_{X,\eta} \subset W(U)_{X,\eps}$ is a basic neighbourhood of $R$. Moreover, if we fix two basic neighbourhoods $W(U)_{X_1,\eps_1},W(U)_{X_2,\eps_2}$ of $U$, we can find $$W(U)_{X_1 \cup X_2,\min\{\eps_1,\eps_2\}} \subset W(U)_{X_1,\eps_1} \cap W(U)_{X_2,\eps_2}.$$ 
However, it can be easily seen that $\tau_{pr}$ is homeomorphic to a pointwise convergence topology of the subspace 
 $$\{(R,D_{R}) \colon R \in \mathcal{R}(\U)\}$$ of the space 
 $$\{(f,g) \colon f \in C(\U,\U),g \in C(\U,\R)\}=C(\U \times \U,\U \times \R)$$ of all continuous functions $\U \times \U \to \U \times \R.$

Furthermore, it is easy to observe that
$$\tau_p \subset \tau_{pr} \subset \tau_u.$$
Since an inlcusion $\tau_{p} \subset \tau_{pr}$ is immediate, we will only check that $\tau_{pr} \subset \tau_u.$ Take $U \in \mathcal{U}(\U)$ and a neighbourhood $W(U)_{X,\eps}$ of $U$ in $\tau_{pr}.$  Then, we have $V(U)_{\frac{\eps}{2}} \subset W(U)_{X,\eps}.$

If $F_{\U}$ is a retract of some universal, ultrahomogeneous retraction $U\colon \U \to \U$, then we define $\mathcal{R}(\U,F_{\U})$ and $\mathcal{U}(\U,F_{\U})$ as the subspaces of $\mathcal{R}(\U),\mathcal{U}(\U)$, consisting of retractions onto the retract $F_{\U}.$
It is easy to observe that
\begin{rem} \label{p=pr}
Topology $ \tau_{pr}$ coincides with  $\tau_p$ on $\mathcal{R}(\U,F_{\U}).$
\end{rem}

Before we start analyzing topological properties of $\mathcal{U}(\U)$ inside of $\mathcal{R}(\U)$ let us observe that the family 
$$\{R[\U] \colon R \in \mathcal{R}(\U)\} \setminus \{U[\U] \colon U \in \mathcal{U}(\U)\}$$ is rich. Recall that $(X,d)$ is $1$-LAR (acronym formed from \emph{1-Lipschitz Absolute Retract}) metric space if for any other metric space $(Y,d_Y)$ with $d_Y \restriction X \times X=d$ there is a $1$-Lipschitz retraction $R \colon Y\to X$ onto $X$ (see more in \cite{BBGNP}). Now, take any separable $1-$LAR metric space $X$ and use a universality of $\U$ to find an isometry $i \colon X\to \U.$ Then, $i[X]$ is $1-$LAR, again. By a definition, there is a $1-$Lipschitz retraction $R \colon \U \to \U$ with $R[\U]=i[X].$ For instance, for any $n \in \N$ the set $[0,1]^n$ considered with a metric $d_{max}$ is $1-$LAR. However, it is known that $\U$ is not $1-$LAR space (see \cite{AI}).


\begin{prop} \label{inter}
$$\mathcal{U}(\U)=\bigcap_{\substack {A \subset \U_0, \ |A|<\infty, \ b \in \U_0, \ n \in \N\\  (A \cup \{b\},r,p) - \mbox{ rational triple }\\ i \colon A \to \U_0  \mbox{ - isometry} \\ }} \bigcup_{x \in A} (E_{A,r,i,p,b,n,x}^c \cup F_{A,r,i,p,b,n,x}^c) $$ $$ \cup \bigcup_{\substack {i':A\cup \{b\}\to \U_0 \mbox{ - isometry with } d_{\U_0}(i'(x),i(x))<\frac{2}{n} \\ \mbox{ for any  } x \in A}}$$ $$ \left[ \bigcap_{x \in A}  (E_{A,r,i,p,b,n,x} \cap F_{A,r,i,p,b,n,x}) \cap \bigcap_{x \in A \cup \{b\}} B_{A,r,i',p,b,n,x} \cap C_{A,r,i',p,b,n,x}\right],$$
	where 
	$$B:=B_{A,r,i',p,b,n,x}=\left\{U \in \mathcal{R}(\U) \colon d_{\U}(U \circ i'(x), i' \circ r(x))<\frac{5}{n} \right\},$$
	$$C:=C_{A,r,i',p,b,n,x}=\left\{U \in \mathcal{R}(\U) \colon |D_U \circ i'(x) - p(x)|<\frac{3}{n}\right\},$$
	$$E:=E_{A,r,i,p,b,n,x}=\left\{U \in \mathcal{R}(\U) \colon d_{\U}(U \circ i(x), i \circ r(x))<\frac{1}{n} \right\},$$
	$$F:=F_{A,r,i,p,b,n,x}=\left\{U \in \mathcal{R}(\U) \colon |D_U \circ i(x) - p(x)|<\frac{1}{n} \right\}.$$
\end{prop}
\begin{proof}
It follows from Theorem \ref{aprox}.
\end{proof}

\begin{prop} \label{gd}
	
	\begin{enumerate}[(i)]
		
	\item 	$\mathcal{U}(\U)$ is of type $G_{\delta\sigma\delta}$ in $\mathcal{R}(\U)$ in $\tau_p$.

    \item  $\mathcal{U}(\U)$ is of type $G_{\delta}$ in $\mathcal{R}(\U)$ in $\tau_{pr}$ and $\tau_u$.

    \item 	$\mathcal{U}(\U,F_{\U})$ is of type $G_{\delta}$ in $\mathcal{R}(\U,F_{\U})$ in $\tau_p=\tau_{pr}$ and $\tau_{u}$.
	
	\end{enumerate}
	
\end{prop}

\begin{proof}
	
    We use a description of $\mathcal{U}(\U)$ in a Proposition \ref{inter}.

	
	
	Ad. $(i)$
    
    We firstly check openness of the sets
	$B$ and $E.$  Hence, fix $A,r,i,i',p,b,n,x$ and begin with $B$.
	
	Fix $U \in B$ and $\varepsilon=\frac{5}{n}-d_{\U}(U \circ i'(x), i' \circ r(x)).$ Since $U \in B,$ we have $\varepsilon>0.$ Then the basic set $V=\{R \in \mathcal{R}(\U) \colon d_{\U}(R(i'(x)),U(i'(x)))<\varepsilon\}$ is contained in $B$ since  
	$$d_{\U}(R(i'(x)),i'(r(x))) \leq d_{\U}(R(i'(x)),U(i'(x)))+d_{\U}(U(i'(x)),i'(r(x))) $$
	$$< \varepsilon+d_{\U}(U(i'(x)),i'(r(x)))=\frac{5}{n}.$$
	
	Now, we will check $F$ is of type $G_{\delta \sigma}$ (analogous verification works for $C$). Enumerate $\U_0=\{x_k \colon k \in \N\}.$
	We have
	$$F=\bigcup_{k \in \N} F_k \cap \bigcup_{m \in \N} \bigcap_{k \in \N} G_{k,m},$$ where  
	$$F_k=\left\{R \in R(\U) \colon d_{\U}(i(x),R(x_k)) <p(x)+\frac{1}{n}\right\} $$ and 
	$$G_{k,m}=\left\{R \in R(\U) \colon p(x)-\frac{1}{n}+\frac{1}{m}<d_{\U}(i(x),R(x_k)) \right\}$$
	for $ k,m \in \N.$ Now, fix $k \in \N$ and check that $F_k$ is open (verification that $G_{k,m}$'s are open is analogous). By a continuity of a metric $d_{\U}$ there is a $\delta>0$ such that for all $y \in \U$ with $d_{\U}(y,R(x_k))<\delta$ we have $|d_{\U}(i(x),y) - p(x)|<\frac{1}{n}.$ Therefore for any $U \in R(\U)$ with $d_{\U}(R(x_k),U(x_k))<\delta$ we have $|d_{\U}(i(x),U(x_k)) - p(x)|<\frac{1}{n}.$ Consequently, $F$ is of type $G_{\delta \sigma}.$
	
	Similarly, we have 
	$$F^c=\bigcap_{k,m \in \N} F_{k,m}\cup \bigcup_{k \in \N} \bigcap_{m \in \N} H_{k,m},$$
	where $$F_{k,m}=\{R \in R(\U) \colon d_{\U}(R(x_k),i(x))>p(x)+\frac{1}{n}-\frac{1}{m}\}$$ and
$$H_{k.m}=\{R \in R(\U) \colon d_{\U}(R(x_k),i(x))<p(x)-\frac{1}{n}+\frac{1}{m}\},$$
 where $k,m \in \N.$ Openness of the sets $F_{k,m}, H_{k,m}$ can be checked analogously to the openness of the sets $F_k.$ Consequently, $F^c$ is of type $G_{\delta \sigma}.$
	
    Moreover, we have
	$$E^c=\bigcap_{m \in \N} E_{m},$$ for the open sets $E_m=\{R \in R(\U) \colon d_{\U}(R \circ i(x), i \circ r(x))>\frac{1}{n}-\frac{1}{m}\}, \ m \in \N.$  This shows that $E^c$ is of type $G_{\delta}$ and finally, $\mathcal{U}(\U)$ is of type $G_{\delta \sigma \delta}$.

Ad. $(ii)$

Since $\tau_{pr} \subset \tau_u$ it suffices to check that $\mathcal{U}(\U)$ is of type $G_{\delta}$ with respect to $\tau_{pr}$. Firstly, we have to note that $B,C,E,F$ are open in $\tau_{pr}.$ We will verify only that $F \in \tau_{pr}.$ By a continuity of algebraic operations we can find $\eta>0$ with 
$$|D_U(i(x))-z|<\eta \Rightarrow |z-p(x)|<\eps.$$ Then we have $W(U)_{\{x\},\eta} \subset F.$ 

Now, we will check that $E^c,F^c$ are of type $G_{\delta}.$ Indeed, we have
$F^c= \bigcap_{m \in \N} H_m,$ where $H_m=\{U \in \mathcal{R}(\U) \colon |D_U(i(x))-p(x)|>\frac{1}{n}-\frac{1}{m}\}$ for $m \in \N.$ Note that all sets $H_m$'s are open - reasoning is similar to the verification of the openness of $F$. Consequently, the set $F^c$ is of type $G_{\delta}$ and finally, the set $\mathcal{U}(\U),$ too.

Ad.$(iii)$

    We consider the same description $\mathcal{U}(\U,F_{\U})$ as for $\mathcal{U}(\U)$, but the sets $B,C,E,F$ consist of $U \in \mathcal{R}(\U,F_{\U}).$ As a consequence $F$ and $C$ are either $\emptyset$ or $\mathcal{R}(\U,F_{\U})$ and $\mathcal{U}(\U,F_{\U})$ is of type $G_{\delta}.$
	
\end{proof}


Now, we will focus on the problem of density of $\mathcal{U}(\U)$ in $\mathcal{R}(\U)$.

\begin{tw} \label{dp}
The set $\mathcal{U}(\U)$ is dense in the space $\mathcal{R}(\U)$ with a pointwise retract convergence topology $\tau_{pr}$ and a pointwise convergence topology $\tau_{p}$.
\end{tw}

\begin{proof}
Since $\tau_{pr} \subset \tau_p$ it suffices to consider a topology $\tau_{pr}.$
Fix any $1-$Lipschitz retraction $T\in \mathcal{R}(\U)$ and any neighbourhood 
$$\mathcal{W}(T)_{X,\eps}=\{S\in  \mathcal{R}(\U): d(S(x_i),T(x_i))< \eps \mbox{ and } |D_S(x)-D_T(x)|<\eps, x_1,\dots, x_n\in X\}$$ of $T$ in $(\mathcal{R}(\U),\tau_{pr})$ for some $X=\{x_1,x_2,\ldots,x_n\} \subset \U, \ \eps>0.$

Let us define $Z=X \cup T[X].$ Then $T[Z] \subset Z.$ Consider the $1-$Lipschitz retraction $U:\U\to U[\U] \subset \U$ satisfying condition $(UR)$. Then using Theorem \ref{URuu} $U$ is universal. So there exist isometric embedding $i:Z\to \U$  such that $U\circ i(x)=i\circ T(x)$ and $\dist(x,T[Z])=\dist(i(x),U[\U])$ for any $x \in Z$.  Again using universality of $U$ there exists isometry $\bar{i}:\U\to \U$ such that $U\circ \bar{i}(x)=\bar{i}\circ T(x)$ and $\dist(x,T[\U])=\dist(\bar{i}(x),U[\U])$ for any $x \in \U$ extending $i$. 
From ultrahomogeneity of the Urysohn space (\cite{M}) there exists an autoisometry $I \colon \U \to \U$ extending $i$. Then we define $S=I^{-1}\circ U\circ I$, by Theorems \ref{retauto} and \ref{URuu} $S$ satisfies (UR). 
Moreover, $S \in \mathcal{W}(T)_{X,\eps}$, since 
\begin{align*}
d(S(x),T(x))&=d(I^{-1}\circ U\circ I(x), T(x))=d(I^{-1}\circ U\circ i(x), T(x))=0<\eps
\end{align*}
and 
\begin{align*}
&|D_S(x)-D_T(x)|=|\mbox{dist}(x,S[\U])-\mbox{dist}(x,T[\U])|= |\mbox{dist}(x,I^{-1}\circ U\circ I[\U])-\mbox{dist}(x,T[\U])|\\
&=|\mbox{dist}(I(x),U\circ I[\U])-\mbox{dist}(x,T[\U])|= |\mbox{dist}(I(x),U[\U])-\mbox{dist}(x,T[\U])|\\&=|\mbox{dist}(i(x),U[\U])-\mbox{dist}(x,T[\U])|=|\mbox{dist}(\bar{i}(x),U[\U])-\mbox{dist}(x,T[\U])|=0<\eps
\end{align*}
for any $x \in X.$
\end{proof}

As a consequence of Proposition \ref{gd} and Theorem \ref{dp} we get following:
\begin{tw}
\begin{enumerate}[(i)]
\item The set $\mathcal{U}(\U)$ is a dense $G_{\delta \sigma \delta}$ subset of $\mathcal{R}(\U)$ in a pointwise convergence topology.
\item The set $\mathcal{U}(\U)$ is a dense $G_{\delta}$ subset of $\mathcal{R}(\U)$ in a pointwise retract convergence topology.
\end{enumerate}
\end{tw}


\begin{prop}
	The set $\mathcal{U}(\U)$ is non-dense in the space $\mathcal{R}(\U)$ with a uniform convergence topology $\tau_u$.
\end{prop}

\begin{proof}
	Consider a constant $1-$Lipschitz retraction $T \colon \U \to \U, \ T \equiv c$ for some $c \in \U.$ Suppose on the contrary that for some $\varepsilon>0$ there is $U \in \mathcal{U}[\U]$ with 
	$$\forall_{x \in \U} \ d_{\U}(U(x),c)<\varepsilon.$$ Therefore $U[\U]$ is bounded. Recall that there is $R \in \mathcal{U}[{\U}]$ with an unbounded retract $R[\U]$ (see [\cite{d},Theorem 2.8.]). Then by a definition of universality of $U$ we have
	$$\forall_{x \in \U} \ i(R(x))=U(i(x))$$
	for some isometry $i \colon \U \to \U.$ Consequently, 
	$\{U(i(x)) \colon x \in \U\}$ is bounded, while \newline 
	$\{i(R(x)) \colon x \in \U\}$ is unbounded, contradiction.
\end{proof}

One can consider only $1-$Lipschitz retractions acting onto retracts of universal, ultrahomogeneous retractions.
Denote by $T=\{U[\U] \colon U \in \mathcal{U}[\U]\}.$
Both results are true in all topologies $\tau_p,\tau_{pr},\tau_u$.
\begin{prop} \label{noa}
	Either for all $t \in T$ the set $\mathcal{U}(\U,t)$ is dense in $\mathcal{R}(\U,t)$ or for all $t \in T$ the set $\mathcal{U}(\U,t)$ is non-dense in $\mathcal{R}(\U,t).$
\end{prop}
\begin{proof}
	Fix $t,s \in T$  and respective $T,S \in \mathcal{U}[\U]$ such that $T[\U]=t, \ S[\U]=s.$ By Theorem \ref{retauto} we can fix an autoisometry $I \colon \U \to \U$ with $I \circ T=S \circ I, \ \dist(x,t)=\dist(I(x),s)$ for all $x \in \U$. Take any retractions $R_s \in \mathcal{R}(\U,s), \ V \in \mathcal{U}(\U,t)$ and note that by Remark \ref{rnr} we have $W=I \circ V \circ I^{-1} \in \mathcal{U}(\U,s)$  and $R_t:=I^{-1} \circ R_s \circ I\in \mathcal{R}(\U,t).$ Observe that for any $x \in \U$ we have
    $$d_{\U}(V(x),R_t(x))=d_{\U}(I^{-1} \circ I \circ V \circ I^{-1} \circ I(x),I^{-1} \circ R_s \circ I(x))=$$ 
    $$=d_{\U}(I \circ V \circ I^{-1}(x),R_s(x))=d_{\U}(W(x),R_s(x)).$$ This equality, together with the definitions of the topologies $\tau_p=\tau_{pr}$ (by Remark \ref{p=pr}) and $\tau_u$, show that a neighbourhood of $R_t$ is disjoint from $\mathcal{U}(\U, t)$ if and only if the corresponding neighbourhood of $R_s$ is disjoint from $\mathcal{U}(\U, s)$.
\end{proof}

\begin{cor}
The set $\mathcal{U}(\U)$ is dense in $\bigcup_{t \in T} \mathcal{R}(\U,t)$ if there is $t \in T$ such that the set $\mathcal{U}(\U,t)$ is dense in $\mathcal{R}(\U,t)$ is dense.
\end{cor}
\begin{proof}
Suppose that there is $t_0 \in T$ with $\mathcal{R}(\U,t_0) \subset \overline{\mathcal{U}(\U,t_0)} .$ By Proposition \ref{noa} we have $\mathcal{R}(\U,t) \subset \overline{\mathcal{U}(\U,t)}$ for all $t \in T.$ Then we have
$$\bigcup_{t \in T} \mathcal{R}(\U,t)\subset \bigcup_{t \in T} \overline{\mathcal{U}(\U,t)} \subset \overline{\bigcup_{t \in T} \mathcal{U}(\U,t)}=\overline{\mathcal{U}(\U)} \subset \bigcup_{t \in T} \mathcal{R}(\U,t).$$ As a result, $\overline{\mathcal{U}(\U)}=\bigcup_{t \in T} \mathcal{R}(\U,t).$
\end{proof}
These results suggest the following problem

\begin{pyt}
\begin{enumerate}[(i)]
\item Is it true that $\mathcal{U}(\U,t)$ is dense in $\mathcal{R}(\U,t)$ in $\tau_{pr}$ for any $t \in T$?
\item Is it true that $\mathcal{U}(\U)$ is dense in $\bigcup_{t \in T} \mathcal{R}(\U,t)$ with $\tau_u$?
\end{enumerate}
\end{pyt}

\end{document}